\documentclass[11pt,a4paper]{article}
\usepackage[T1]{fontenc}
\usepackage[a4paper]{geometry}
\usepackage{amssymb,latexsym,amsmath,amsfonts,amsthm}
\usepackage{graphicx}
\usepackage{epstopdf}
\usepackage{caption}
\usepackage{algorithm}
\usepackage{algorithmic}

\usepackage{dsfont}
\usepackage{mathrsfs}
\usepackage{epsfig}
\usepackage{booktabs}
\usepackage{mathtools}
\usepackage{comment,verbatim}
\usepackage{subfigure}
\usepackage{overpic}
\usepackage{hyperref}
\usepackage{bbm}
\usepackage[normalem]{ulem}
\usepackage{bm}

\newtheorem{theorem}{Theorem}[section]

\newtheorem{lemma}[theorem]{Lemma}

\makeatletter 
\@addtoreset{equation}{section}
\makeatother  

\newtheorem{remark}[theorem]{Remark}

\graphicspath{ {images/} }

\usepackage{color}

\begin{document}

\title{An efficient spectral method for the fractional Schr\"{o}dinger equation on the real line}

\author{Mengxia Shen\footnotemark[1] ~ and ~ Haiyong Wang\footnotemark[1]~\footnotemark[2]}

\footnotetext[1]{School of Mathematics and Statistics, Huazhong University of Science and Technology, Wuhan 430074, P. R. China. \texttt{Email:haiyongwang@hust.edu.cn}}

\footnotetext[2]{Hubei Key Laboratory of Engineering Modeling and Scientific Computing, Huazhong University of Science and Technology, Wuhan 430074, P. R. China}

\maketitle

\begin{abstract}
The fractional Schr\"{o}dinger equation (FSE) on the real line arises in a broad range of physical settings and their numerical simulation is challenging due to the nonlocal nature and the power law decay of the solution at infinity. In this paper, we propose a new spectral discretization scheme for the FSE in space based upon Malmquist-Takenaka functions. We show that this new discretization scheme achieves much better performance than existing discretization schemes in the case where the underlying FSE involves the square root of the Laplacian, while in other cases it also exhibits comparable or even better performance. Numerical experiments are provided to illustrate the effectiveness of the proposed method.
\end{abstract}


\vspace{0.05in}


\section{Introduction}
Nonlocal and fractional models have become increasingly important because of their connection with many real-world phenomena that appear in physics, biology and materials science. In contrast to local models, nonlocal and fractional models are more suitable to model complex systems exhibiting singularities and anomalies as well as involving nonlocal interactions (see, e.g., \cite{du2019}). The fractional Laplacian operator, which can be seen as the infinitesimal generator of a symmetric $\alpha$-stable L\'{e}vy process in probability theory, is one of the fundamental nonlocal operators and it arises in a number of applications such as anomalous diffusion, image denoising, finance. The fractional Laplacian operator of a function $f:\mathbb{R}^d\rightarrow\mathbb{R}$ is defined by
\begin{equation}\label{def:FLCPV}
(-\Delta)^{\alpha/2}f(x) := \alpha\frac{{2}^{\alpha-1}\Gamma(\frac{\alpha+d}{2})}{\pi^{d/2}\Gamma{(\frac{2-\alpha}{2}})} ~\mathrm{p.v.} \int_{\mathbb{R}^d}\frac{f(x) - f(y)}{|x-y|^{d+\alpha}} \mathrm{d}y, \quad  x\in{\mathbb{R}^d},
\end{equation}
where $\alpha\in{(0,2)}$ and $\mathrm{p.v.}$ stands for the Cauchy principle value. Equivalently, it can also be defined as a pseudo-differential operator via the Fourier transform
\begin{equation}\label{def:FLFT}
(-\Delta)^{\alpha/2} f(x) := \mathcal{F}^{-1}[|\xi|^{\alpha} \mathcal{F}[f](\xi)](x),
\end{equation}
where $\mathcal{F}$ and $\mathcal{F}^{-1}$ denote the Fourier and inverse Fourier transforms, respectively. It is well known that the fractional Laplacian operator reduces to the identity operator whenever $\alpha\rightarrow0$ and to the negative Laplacian operator whenever $\alpha\rightarrow2$. When discretizing nonlocal models involving the fractional Laplacian operators, the main difficulty stems from the nonlocal and singular nature of the fractional Laplacian operators and the slow decay of the underlying solution at infinity. In the past decade, numerical methods for such nonlocal models have attracted a lot of attention and many significant advances have been made, such as the finite element methods \cite{Acosta2017,Bonito2019}, the finite difference methods \cite{Duo2018,Huang2014,Minden2020} and spectral methods \cite{Acosta2018,Cayama2020AMC,Mao2017,Sheng2020MCF,Sheng2021,Tang2020}. 
Among these methods, the finite element and finite difference methods are typically studied for nonlocal models on bounded domains. While for nonlocal models on unbounded domains, spectral methods are particularly attractive due to their global character.

In this paper, we are interested in the fractional Schr\"{o}dinger equation (FSE) on the real line
\begin{align}\label{eq:Model}
\begin{cases}
\mathrm{i}\partial_{t} \psi(x,t) = \gamma(-\Delta)^{\alpha/2} \psi(x,t) + \mathcal{T}\psi(x,t), \quad  x\in\mathbb{R}, \quad t>0,  \\[5pt]
\psi(x,0) = \psi_0(x),  \quad  x\in\mathbb{R}, \\[3pt]
\lim_{|x|\rightarrow{\infty}}\psi(x,t) = 0,
\end{cases}
\end{align}
where $\mathrm{i}=\sqrt{-1}$, $\gamma\in\mathbb{R}$ and $\gamma\neq0$, $\mathcal{T}$ is a linear or nonlinear operator (e.g., $\mathcal{T}\psi(x,t)=V(x)\psi(x,t)$ or $\mathcal{T}\psi(x,t)=\pm|\psi(x,t)|^2\psi(x,t)$) and $\psi(x,t)$ is a complex-valued wave function. Equation \eqref{eq:Model}, which was introduced by Laskin in \cite{laskin2002}, is a natural generalization of the standard Schr\"{o}dinger equation that arises in the context of the well-known Feynman path integrals approach to quantum mechanics when the Brownian trajectories are replaced by L\'{e}vy flights. More recently, experimental realizations and rigorous derivations of FSE have been widely investigated in different branches of physics, such as the continuum limit of certain discrete physical systems with long-range interactions \cite{kirk2013}, beam propagation \cite{Huang2017,longhi2015} and the L\'{e}vy crystal in a condensed-matter environment \cite{stickler2013}. Around the same time, numerical methods that combine spectral methods for the spatial discretization and time-stepping methods for the temporal discretization for the solution of \eqref{eq:Model} have attracted considerable attention (see, e.g., \cite{Duo2016,Klein2014,Mao2017,Sheng2020MCF,Sheng2021}). However, when the solution decays slowly with a power law at infinity, exponential convergence of these existing spectral discretization schemes using either Fourier, Hermite or mapped Chebyshev functions in space has so far not been observed. Here, we propose a novel spectral discretization scheme using Malmquist-Takenaka functions which have excellent approximation properties for functions with poles in the complex plane. We show that this new discretization method achieves exponential convergence in the particular case of $\alpha=1$, regardless of the underlying FSE is linear or nonlinear, and exhibits a comparable or even better performance than state-of-the-art discretization schemes in other cases.

The rest of this paper is organized as follows. In section \ref{sect:MTfun}, we briefly review some properties of Malmquist-Takenaka functions that will be used for the construction of our spectral discretization method. In section \ref{sect:MTSM}, we present a numerical method for FSE by combining a spectral Galerkin method using Malmquist-Takenaka functions for the space discretization coupled with time-stepping schemes for the temporal discretization. We perform two numerical examples to illustrate the performance of the proposed method in section \ref{sect:Exam} and conclude the paper with some remarks in section \ref{sect:Con}.

\section{Malmquist-Takenaka functions}\label{sect:MTfun}
The Malmquist-Takenaka functions\footnote{Up to some constant and scaling factors, the Malmquist-Takenaka functions are also known as the Christov functions in some literature. Here we adopt the name used in \cite{Iserles2020FAA}.} (MTFs) are defined by
\begin{equation}\label{def:MT}
\varphi_n(x) = \mathrm{i}^n \sqrt{\frac{2}{\pi}} \frac{(1+2\mathrm{i}x)^n}{(1-2\mathrm{i}x)^{n+1}}, \quad n\in{\mathbb{Z}}.
\end{equation}
Let $L^2(\mathbb{R})$ denote the space of square integrable functions and let $(\cdot,\cdot)$ denote the inner product defined by $(f,g)=\int_{\mathbb{R}}f(x)\overline{g(x)}\mathrm{d}x$, where $\overline{f(x)}$ denotes the complex conjugate of $f(x)$. It is well known that the system $\{\varphi_n\}_{n\in\mathbb{Z}}$ forms a complete and orthonormal basis in $L^2(\mathbb{R})$, i.e.,
\begin{equation}\label{eq:MTOrth}
(\varphi_n,\varphi_m) = \delta_{n,m},
\end{equation}
where $\delta_{n,m}$ is the Kronecker delta. Theoretical aspects of MTFs as well as their applications in designing algorithms for Fourier and Hilbert transforms have been investigated during the past few decades (see, e.g., \cite{Christov1982,Higgins1977,Iserles2020FAA,Iserles2021,Weber1980,Weideman1995}). We list below some of theoretical properties of MTFs that will be used for the construction of spectral method.
\begin{itemize}
\item[(i)] They satisfy the following differential recurrence relation
\begin{equation}\label{eq:DiffI}
\frac{\mathrm{d}}{\mathrm{d}x}\varphi_n(x) = -n \varphi_{n-1}(x) + \mathrm{i}(2n+1)\varphi_n(x) + (n+1)\varphi_{n+1}(x),
\end{equation}
from which we can deduce that the differentiation matrix of the Malmquist-Takenaka system is skew-Hermitian and tridiagonal. Moreover, higher order derivatives of $\varphi_n(x)$ can be derived from repeated application of \eqref{eq:DiffI}. Moreover, they also satisfy
\begin{equation}\label{eq:DiffII}
x\frac{\mathrm{d}}{\mathrm{d}x}\varphi_n(x) = -\frac{n}{2} \mathrm{i}\varphi_{n-1}(x) - \frac{1}{2}\varphi_n(x)-\frac{n+1}{2}\mathrm{i}\varphi_{n+1}(x).
\end{equation}

\item[(ii)] They are uniformly bounded on the real line and
\begin{equation}\label{eq:UB}
|\varphi_n(x)| = \sqrt{\frac{2}{\pi}} \frac{1}{\sqrt{1+4x^2}} \leq \sqrt{\frac{2}{\pi}}, \quad x\in\mathbb{R}.
\end{equation}

\item[(iii)] For $N\in\mathbb{N}$, we define the space $\mathbb{V}_{N}(\mathbb{R}) = \mathrm{span}\{\varphi_{k}(x)\}_{k=-N}^{N-1}$ and denote by $\Pi_N:L^2(\mathbb{R})\rightarrow\mathbb{V}_{N}(\mathbb{R})$ the orthogonal projection operator, i.e.,
\begin{equation}\label{eq:MTExp}
(\Pi_N f)(x) = \sum_{k=-N}^{N-1} a_k \varphi_k(x), \quad  a_k = (f,\varphi_k).
\end{equation}
With the change of variable $x=\tan(\theta/2)/2$, we obtain
\begin{align}\label{eq:MTCoeff}
a_k &= \frac{(-\mathrm{i})^k}{2\sqrt{2\pi}} \int_{-\pi}^{\pi} f\left(\frac{1}{2}\tan\frac{\theta}{2}\right) \left(1-\mathrm{i}\tan\frac{\theta}{2}\right) e^{-\mathrm{i}k\theta} \mathrm{d}\theta \nonumber \\
&\approx \frac{\mathrm{(-i)}^k}{2N} \sqrt{\frac{\pi}{2}} \sum_{j=0}^{2N-1} f\left(\frac{1}{2}\tan\frac{\theta_j}{2}\right) \left(1-\mathrm{i}\tan\frac{\theta_j}{2}\right) e^{-\mathrm{i}k\theta_j},
\end{align}
where $\theta_j=-\pi+{\pi}j/N$ and we have used the fact that integrals of periodic functions can be computed efficiently by using the composite trapezoidal rule. Hence, the computation of $\{a_k\}_{k=-N}^{N-1}$ can be performed rapidly with a single FFT in $\mathcal{O}(N\log N)$ operations.

\item[(iv)] They are eigenfunctions of the Hilbert transform
\begin{equation}\label{eq:Hilb}
\mathcal{H}[\varphi_n](x) = (-\mathrm{i})\mathrm{sgn}(n)\varphi_n(x), \quad \mathcal{H}[f](x) = \frac{1}{\pi} \mathrm{p.v.} \int_{\mathbb{R}}\frac{f(z)}{x-z}\mathrm{d}z,
\end{equation}
where $\mathrm{sgn}(n)=1$ for $n=0,1,\ldots$ and $\mathrm{sgn}(n)=-1$ for $n=-1,-2,\ldots$. Moreover, the Hilbert transform is related to the square root of the Laplacian by $(-\Delta)^{1/2} f(x)=\mathcal{H}[f{'}](x)$.
\end{itemize}

\begin{figure}
\centering
\includegraphics[height=5.cm,width=7cm]{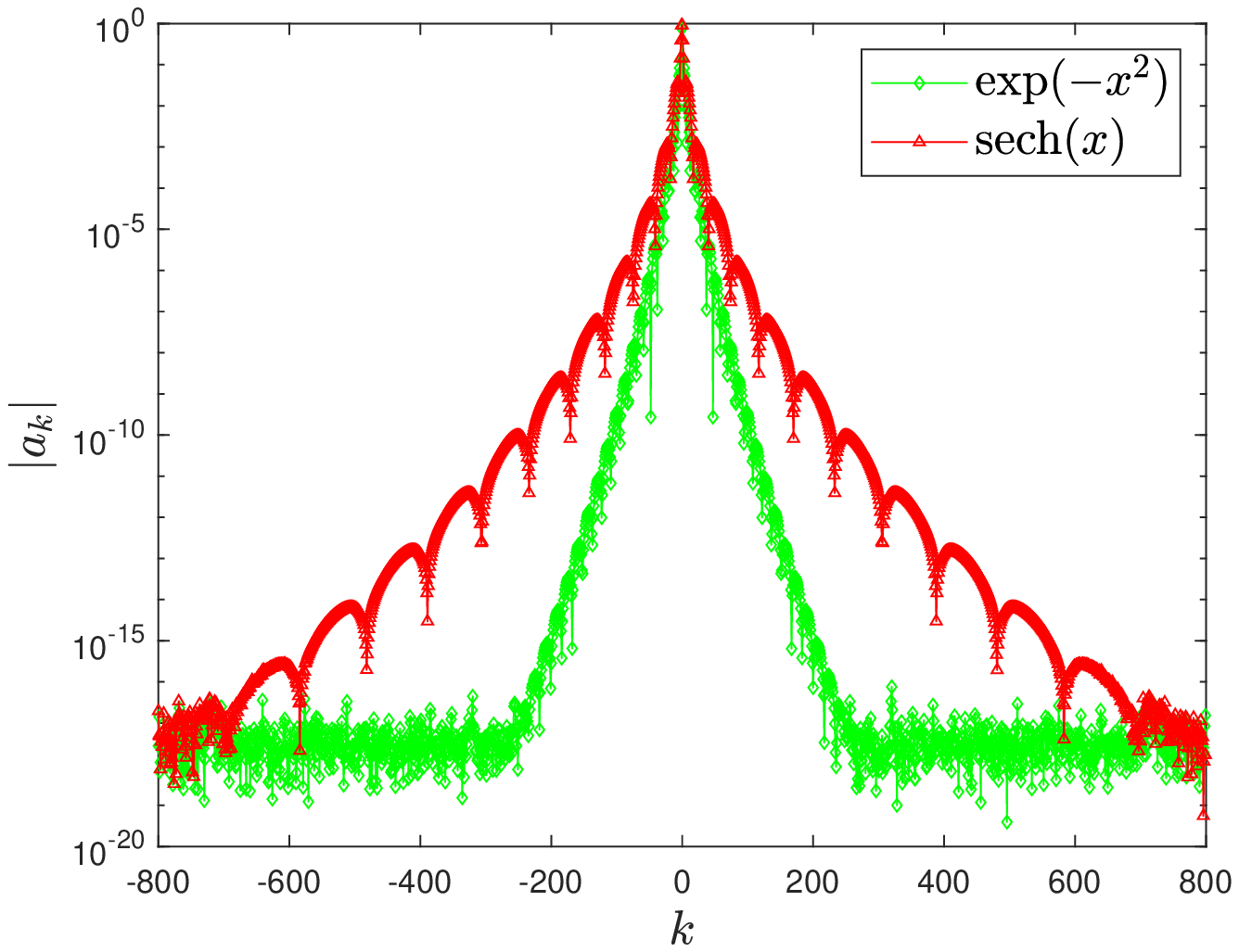}
\includegraphics[height=5.cm,width=7cm]{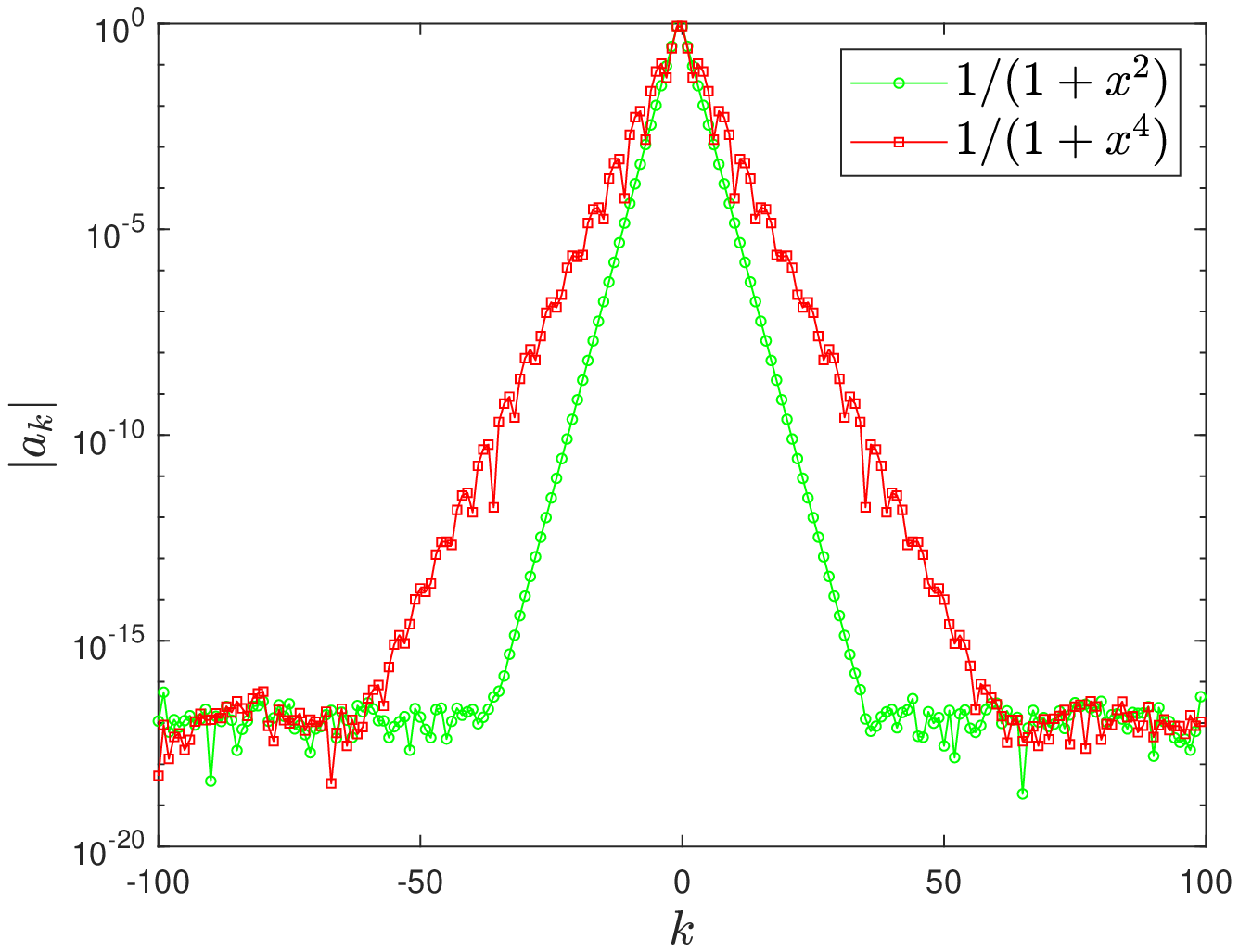}
\caption{The magnitude of the Malmquist-Takenaka coefficients of $\{a_k\}_{k\in\mathbb{Z}}$. Left: $f(x)=\exp(-x^2)$ and $\mathrm{sech}(x)$. Right: $f(x)=1/(1+x^2)$ and $1/(1+x^4)$.}
\label{fig:expalgcoeff}
\end{figure}

Having introduced basic properties of MTFs, we now move on to their spectral approximation properties. To gain insight, we plot in Figure \ref{fig:expalgcoeff} the absolute value of the Malmquist-Takenaka coefficients of $f(x)=\exp(-x^2),\mathrm{sech}(x),1/(1+x^2),1/(1+x^4)$. Clearly, we observe that the coefficients of the first two functions decay at subexponential rates and the coefficients of the last two functions decay at exponential rates. Indeed, Weideman in \cite{Weideman1995} had studied the asymptotic estimate of the coefficients $\{a_k\}_{k\in\mathbb{Z}}$ of several representative functions, including these four test functions, and his analysis provides an important insight to understand the approximation power of MTFs. Below we present a simplified version of Weideman's result on the exponential convergence of Malmquist-Takenaka approximation and we provide a short proof for the purpose of being self-contained.
\begin{theorem}\label{thm:MTRate}
Let $\Pi_Nf$ denote the spectral approximation defined in \eqref{eq:MTExp} and let $\mathcal{A}_{\rho}$ denote the annulus defined by $\mathcal{A}_{\rho}=\{z\in\mathbb{C}:\rho^{-1}<|z|<\rho\}$. Moreover, let
\begin{equation}
\hat{f}(z) = \frac{2}{1+z} f\left(\frac{z-1}{2\mathrm{i}(z+1)}\right).
\end{equation}
If $\hat{f}(z)$ is analytic in $\mathcal{A}_{\rho}$ for some $\rho>1$. Then for $x\in\mathbb{R}$,
\begin{align}\label{eq:MTErrorBound}
\|f - \Pi_Nf \|_{L^{\infty}(\mathbb{R})} =\mathcal{O}(\rho^{-N}).
\end{align}
\end{theorem}
\begin{proof}
By using the map $z=(1+2\mathrm{i}x)/(1-2\mathrm{i}x)$, we have
\begin{align}
a_k &= (f,\varphi_k) = \frac{(-\mathrm{i})^{k+1}}{\sqrt{2\pi}} \oint_{|z|=1} \hat{f}(z) z^{-k} \mathrm{d}z, \nonumber
\end{align}
which implies that $a_k$ is also the Laurent coefficient of $\hat{f}(z)$. Furthermore, since $\hat{f}(z)$ is analytic in the annulus $\mathcal{A}_{\rho}$, it follows that $a_k=\mathcal{O}(\rho^{-|k|})$ for all $k\in\mathbb{Z}$. Consequently,
\begin{align}
\|f - \Pi_Nf \|_{L^{\infty}(\mathbb{R})} \leq \sum_{k=N}^{\infty} |a_k| \|\varphi_k\|_{L^{\infty}(\mathbb{R})}
+ \sum_{k=-\infty}^{-N-1} |a_k| \|\varphi_k\|_{L^{\infty}(\mathbb{R})}, \nonumber
\end{align}
and the desired result \eqref{eq:MTErrorBound} follows by combining the above inequality with \eqref{eq:UB}.
\end{proof}

A few remarks on the approximation power of MTFs are in order.
\begin{remark}
If $f(x)$ has a partial fraction decomposition of the form
\begin{equation}\label{eq:fraction}
f(x) = \sum_{k=1}^{m} \sum_{j=1}^{\chi_k} \frac{\eta_{k,j}}{(x-s_k)^j},
\end{equation}
where $\{s_k\}_{k=1}^{m}$ is a set of poles in the complex plane but not on $\mathbb{R}$ and $\{\chi_k\}_{k=1}^{m}$ is a set of orders associated with those poles, then the above theorem indicates that the Malmquist-Takenaka projection $\Pi_Nf$ converges at an exponential rate.
\end{remark}

\begin{remark}
Recently, Iserles, Luong and Webb in \cite{Iserles2021} compared the approximation power of the Malmquist-Takenaka, Hermite and stretched Fourier functions for Gaussian wave packet functions of the form $f(x) = \exp(-\beta(x-x_0)^2)\cos(\omega{x})$, where $\beta>0$ and $x_0,\omega\in\mathbb{R}$. After some lengthy algebra, they derived the decay rates of the coefficients with respect to these three orthogonal systems, respectively, and concluded that the MTFs are superior to the other two functions. Note that all those three functions have banded skew-Hermitian differentiation matrices.
\end{remark}

\begin{remark}
Mapped Chebyshev functions (MCFs) were recently used to develop spectral methods for PDEs with fractional Laplacian on the unbounded domain \cite{Sheng2020MCF}. Specifically, the MCF is defined by
\begin{equation}
\mathbb{T}_k(x) = \frac{1}{\sqrt{c_k\pi/2}} \frac{1}{\sqrt{1+x^2}} T_k\left( \frac{x}{\sqrt{1+x^2}} \right), \quad k = 0,1,\ldots. \nonumber
\end{equation}
where $c_0=2$ and $c_k=1$ for $k\geq1$ and $T_k(x)$ is the Chebyshev polynomial of the first kind of degree $k$. It is easy to verify that $\{\mathbb{T}_k(x)\}_{k=0}^{\infty}$ forms an orthonormal system on $\mathbb{R}$. Let $S_N^{\mathrm{MC}}(x)$ denote the MCF approximation of the form
\begin{equation}
S_N^{\mathrm{MC}}(x) = \sum_{k=0}^{2N-1} a_k^{\mathrm{MC}} \mathbb{T}_k(x), \quad  a_k^{\mathrm{MC}} = \int_{\mathbb{R}} f(x) \mathbb{T}_k(x) \mathrm{d}x.
\end{equation}
Concerning MTF and MCF approximations, it is natural to ask which one is better. Note that both $\Pi_Nf$ and $S_N^{\mathrm{MC}}(x)$ have the same number of terms. Figure \ref{fig:algfuncom} displays the maximum errors of $\Pi_N(f)$ and $S_N^{\mathrm{MC}}(x)$ for $f(x)=\exp(-x^2)/(1+\mathrm{i}x)$, $\exp(-x^2)$, $1/(x^2+4)$, $\mathrm{sech}(x)$, $1/(x^4+1)$ and $1/(x^4+1)^{1.2}$. Clearly, we can see that both approximations have their own advantages. A systematic study on the comparison of both approximations is beyond the scope of the present paper and will be discussed in another publication.
\end{remark}

\begin{figure}
\centering
\includegraphics[height=5.cm,width=7cm]{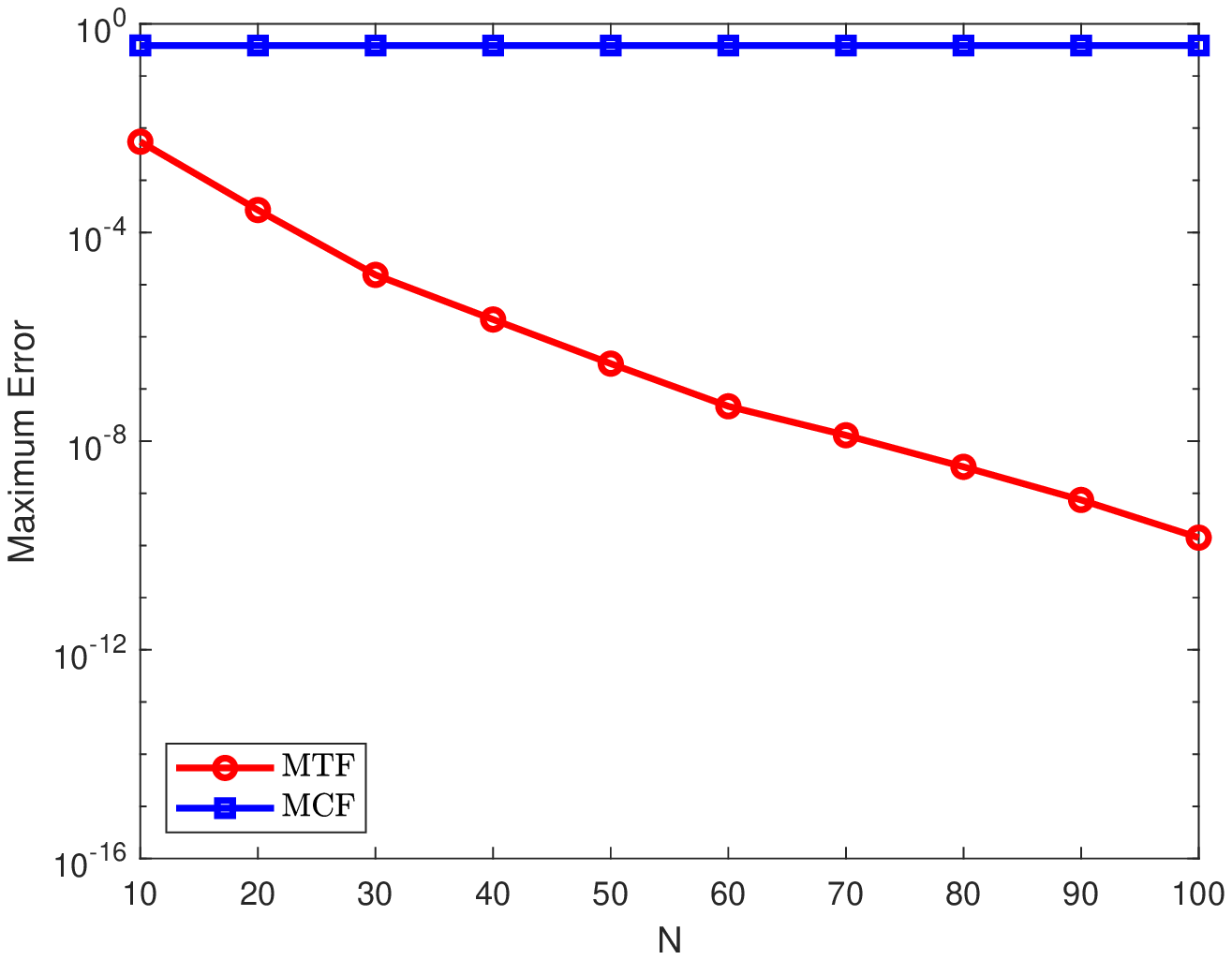}
\includegraphics[height=5.cm,width=7cm]{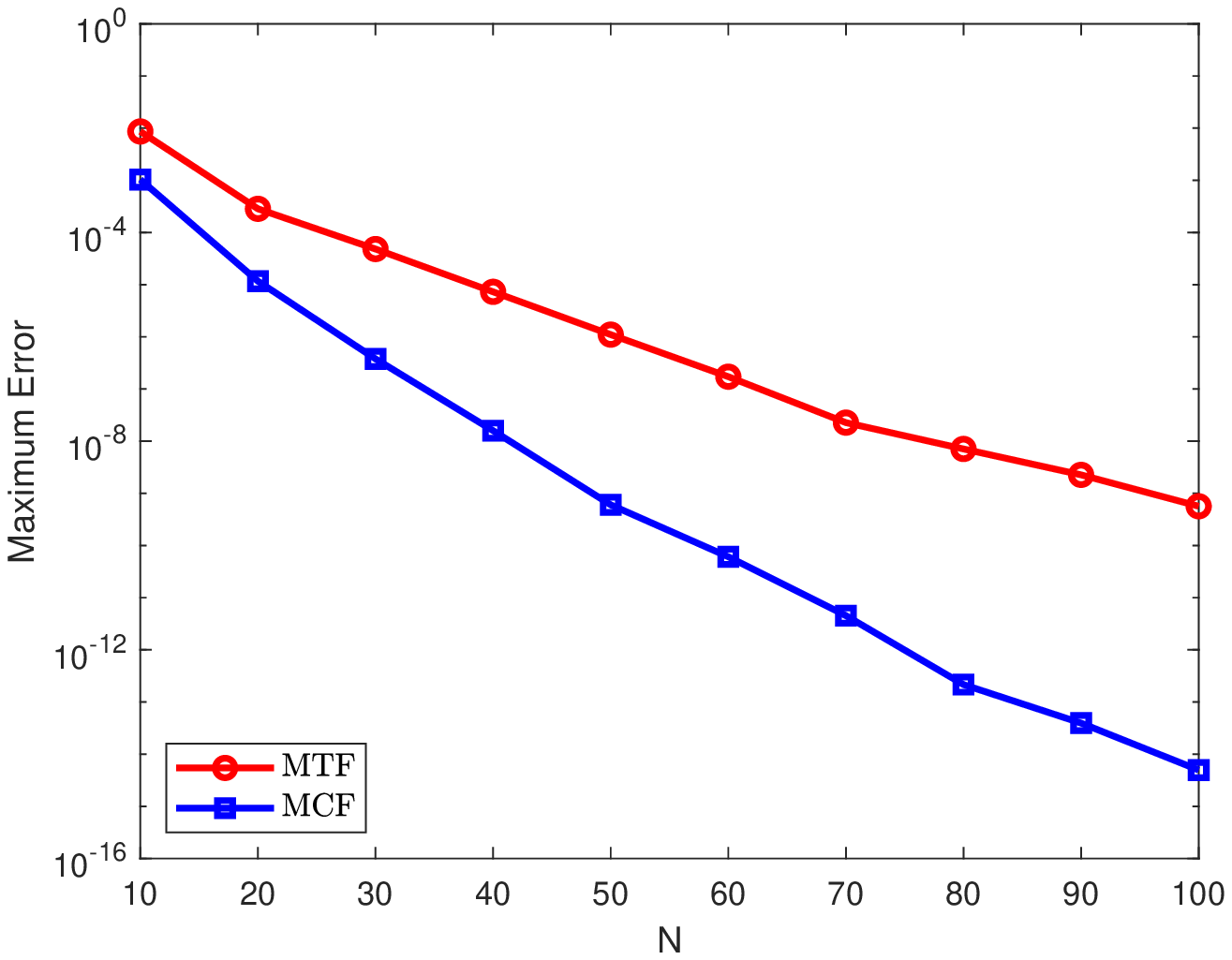}
\includegraphics[height=5.cm,width=7cm]{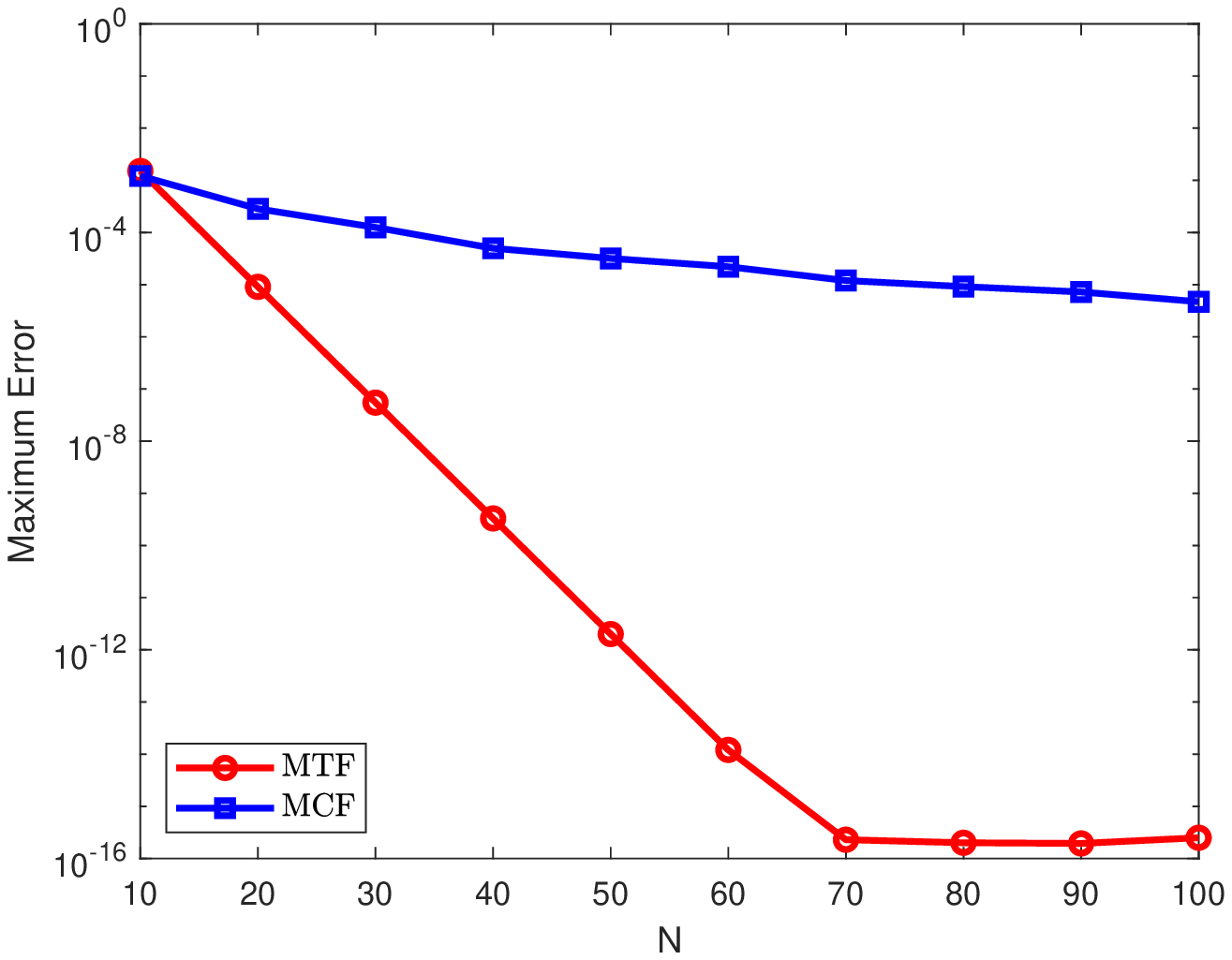}
\includegraphics[height=5.cm,width=7cm]{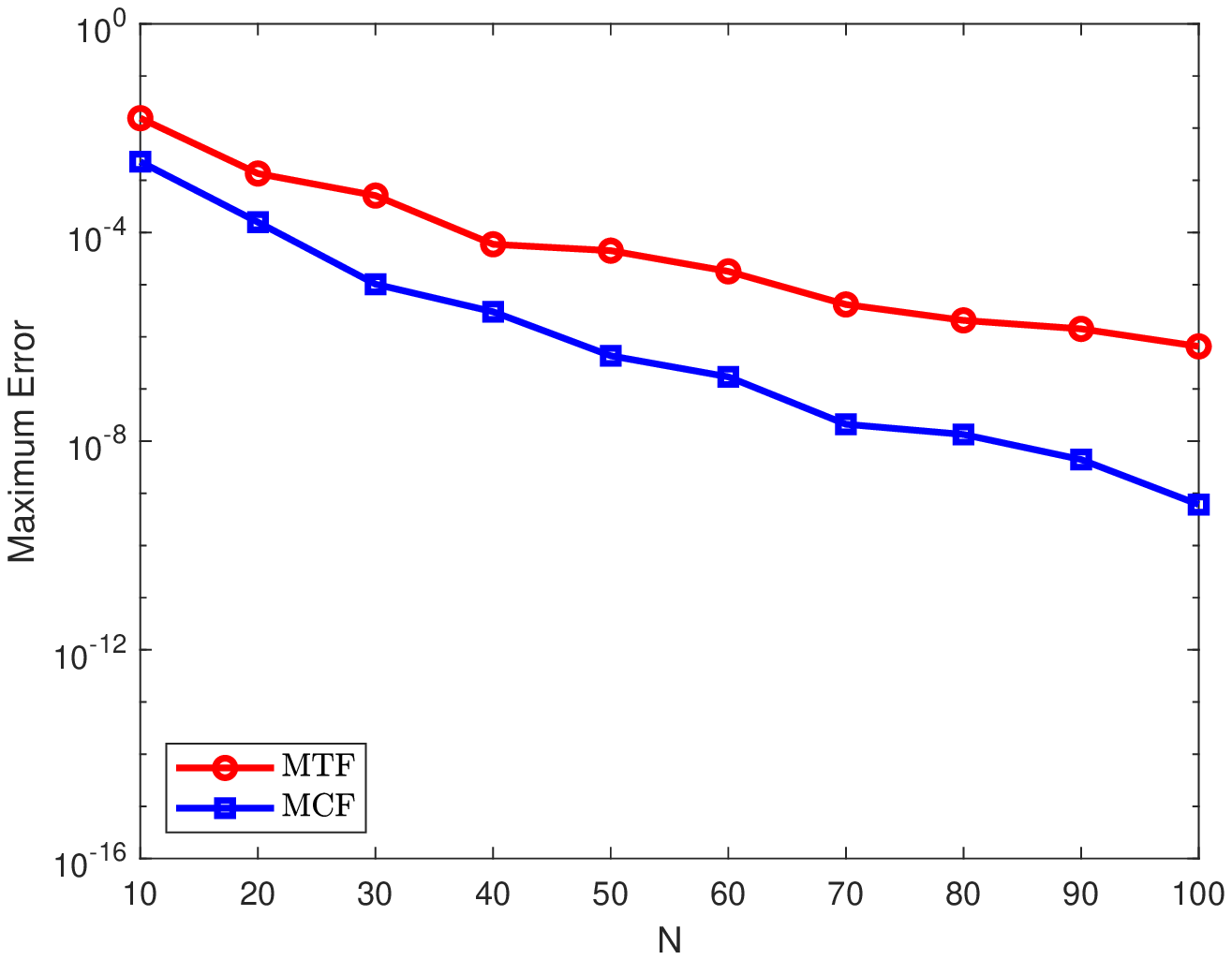}
\includegraphics[height=5.cm,width=7cm]{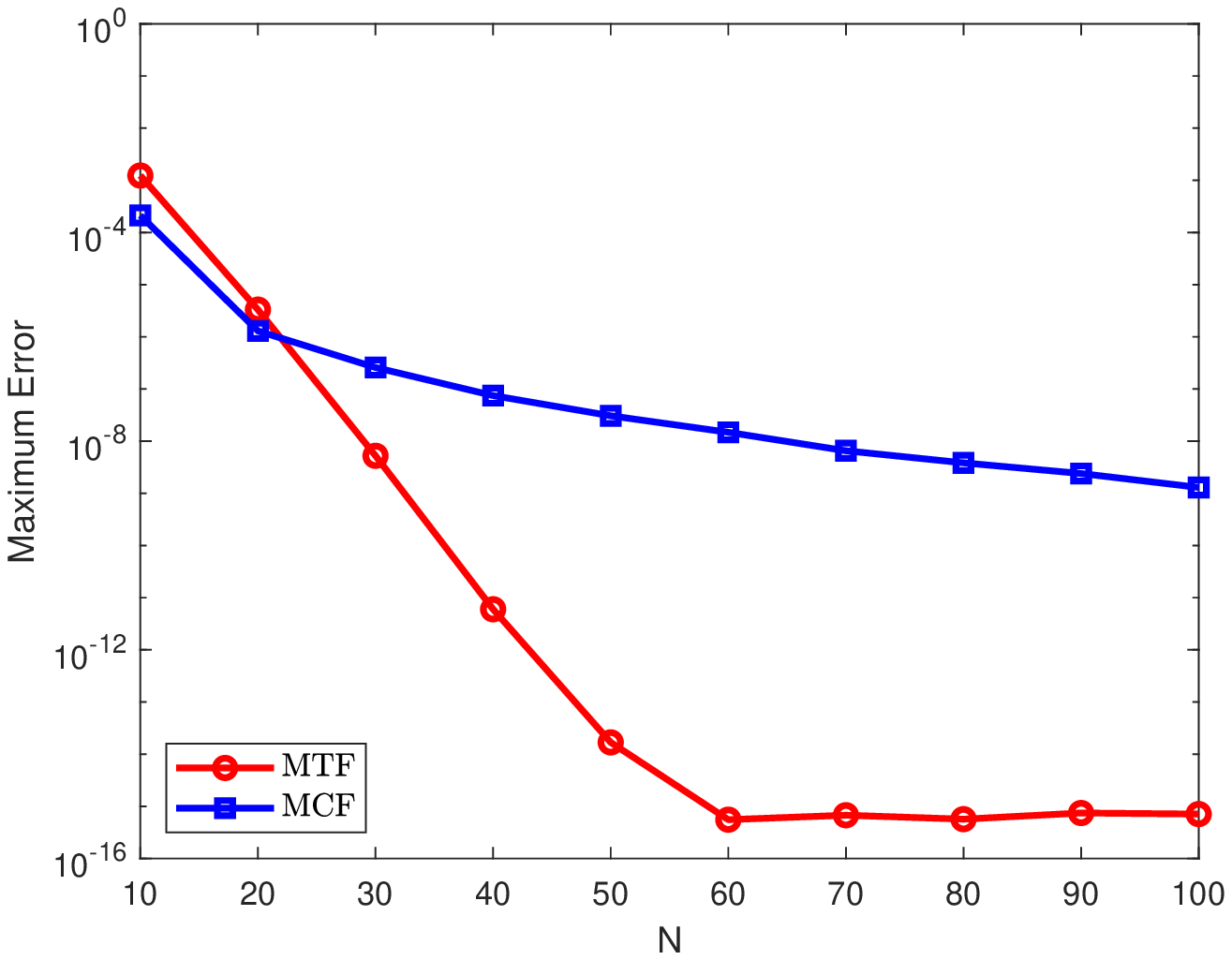}
\includegraphics[height=5.cm,width=7cm]{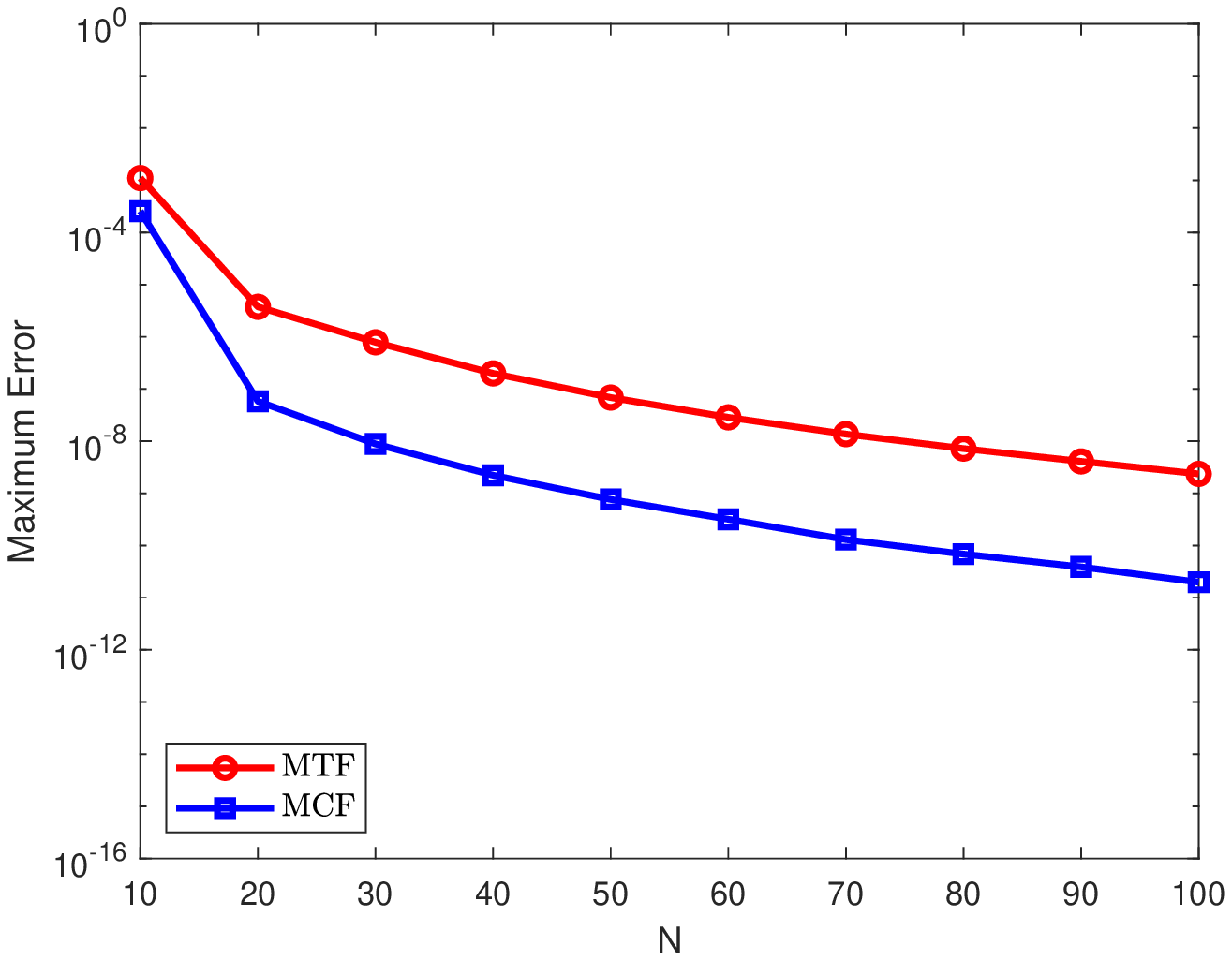}
\caption{Maximum errors of MTF and MCF approximations for $f(x)=\exp(-x^2)/(1+\mathrm{i}x)$, $1/(x^2+4)$, $1/(x^4+1)$ (left column) and $f(x)=\exp(-x^2)$, $\mathrm{sech}(x)$, $1/(x^4+1)^{1.2}$ (right column).}
\label{fig:algfuncom}
\end{figure}

\begin{remark}
In practical calculations, it is beneficial to introduce a scaling parameter in MTFs. More specifically, the scaled MTFs are defined by
\begin{equation}\label{eq:Scaling}
\varphi_n^{\mathrm{S}}(x) = \nu^{-1/2} \varphi_n\left(\frac{x}{\nu}\right), \quad  n\in\mathbb{Z},
\end{equation}
where $\nu>0$ is a scaling parameter. It is easy to check that $\{\varphi_n^{\mathrm{S}}\}_{n\in\mathbb{Z}}$ forms an orthonormal sequence on $\mathbb{R}$. In the rest of this paper, we will use the scaled Malmquist-Takenaka approximations when we mention the scaling parameter $\nu$ explicitly.
\end{remark}

The generalized Laguerre polynomials (GLPs), denoted by $L_n^{(\alpha)}(x)$ with $\alpha>-1$, are orthogonal with respect to the weight function $\omega_{\alpha}(x)=x^{\alpha}e^{-x}$ on the half line $\mathbb{R}_+:=(0,+\infty)$ and
\begin{equation}
\int_{\mathbb{R}_{+}} L_n^{(\alpha)}(x) L_m^{(\alpha)}(x) \omega_{\alpha}(x) \mathrm{d}x = \gamma_n^{(\alpha)}\delta_{n,m},
\end{equation}
where $\gamma_n^{(\alpha)}=\Gamma(n+\alpha+1)/\Gamma(n+1)$. In particular, we will drop the superscript in $L_n^{(\alpha)}(x)$ whenever $\alpha=0$, i.e., $L_n^{(0)}(x)=L_n(x)$. With the Laguerre polynomials we introduce a sequence of functions on the real line for all $n\in\mathbb{Z}$ (see \cite{Keilson1981,Weber1980})
\begin{equation}\label{def:Psi}
\Psi_n(x) = \left\{
\begin{array}{ll}
e^{-x/2}L_n(x)H(x),         &  n\geq0,  \\[8pt]
-e^{x/2}L_{-n-1}(-x)H(-x),  &  n<0,
\end{array}
\right.
\end{equation}
where $H(x)$ is the Heaviside step function. It is easy to check that $\Psi_n(x)=-\Psi_{-n-1}(-x)$ for all $n\in\mathbb{Z}$ and $\{\Psi_n\}_{n\in\mathbb{Z}}$ forms an orthonormal function sequence on the real line.

In the following result, we state the connection between $\{\Psi_n\}_{n\in\mathbb{Z}}$ and $\{\varphi_n\}_{n\in\mathbb{Z}}$ and present an explicit formula for the fractional Laplacian of $\{\varphi_n\}_{n\in\mathbb{Z}}$.
\begin{lemma}\label{lem:FTvarphi}
The Fourier transform of $\varphi_n(x)$ is
\begin{equation}\label{eq:FTvarphi}
\mathcal{F}[\varphi_n](\xi) = (-\mathrm{i})^n \Psi_n(\xi),
\end{equation}
and the fractional Laplacian of $\varphi_n(x)$ is
\begin{equation}\label{eq:FLvarphi}
(-\Delta)^{\alpha/2}\varphi_n(x) = \frac{\Gamma(\alpha+1)}{\sqrt{2\pi}} \left\{
\begin{array}{ll}
{\displaystyle \frac{(-\mathrm{i})^n}{(\frac{1}{2}-\mathrm{i}x)^{\alpha+1}}  ~ _2F_1\left(
\begin{gathered}
-n,\alpha+1\\ 1
\end{gathered}
\middle|\, \frac{1}{\frac{1}{2}-\mathrm{i}x}\right)},         &  n\geq0,  \\[14pt]
{\displaystyle \frac{-(-\mathrm{i})^n}{(\frac{1}{2}+\mathrm{i}x)^{\alpha+1}}  ~ _2F_1\left(
\begin{gathered}
n+1,\alpha+1\\ 1
\end{gathered}
\middle|\, \frac{1}{\frac{1}{2}+\mathrm{i}x}\right)},         &  n<0,
\end{array}
\right.
\end{equation}
where $_2F_1(\cdot)$ is the Gauss hypergeometric function (see, e.g., \cite[Chapter~15]{Olver2010}).
\end{lemma}
\begin{proof}
We only consider the case of $n\geq0$ since the case of $n\leq-1$ can be proved in a similar way. Taking the inverse Fourier transform of $\Psi_n(x)$, we obtain
\begin{align}
\mathcal{F}^{-1}[\Psi_n](x) &= \frac{1}{\sqrt{2\pi}} \int_{\mathbb{R}} \Psi_n(\xi) e^{\mathrm{i}x\xi} \mathrm{d}\xi
= \frac{1}{\sqrt{2\pi}} \int_{\mathbb{R}_{+}} e^{-\xi/2}L_n(\xi) e^{\mathrm{i}x\xi} \mathrm{d}\xi = \mathrm{i}^n \varphi_n(x), \nonumber
\end{align}
where we have used the formula \cite[Equation~(18.17.34)]{Olver2010} in the last step. The desired result \eqref{eq:FTvarphi} follows immediately by taking the Fourier transform of the above equality. As for \eqref{eq:FLvarphi}, using \eqref{def:FLFT} and \eqref{eq:FTvarphi}, we have
\begin{align}
(-\Delta)^{\alpha/2}\varphi_n(x) = (-\mathrm{i})^n \mathcal{F}^{-1}[|\xi|^{\alpha} \Psi_n(\xi)] = \frac{(-\mathrm{i})^n}{\sqrt{2\pi}}
\int_{\mathbb{R}} |\xi|^{\alpha} \Psi_n(\xi) e^{\mathrm{i}\xi x} \mathrm{d}\xi. \nonumber
\end{align}
The desired result \eqref{eq:FLvarphi} follows from applying \cite[Equation~(7.414.7)]{Gradshteyn2007} to the last equation. This ends the proof.
\end{proof}

\begin{remark}
Let $\mu_n(x)=\mathrm{i}^n(\pi/2)^{1/2}\varphi_n(-x/2)$ for $n\in\mathbb{Z}$. The fractional Laplacian of $\mu_n(x)$ was recently derived in \cite[Proposition~3.1]{Cayama2020ANM} based on the techniques of complex analysis. Here we provide an alternative approach for its derivation.
\end{remark}

\begin{remark}
From Lemma \ref{lem:FTvarphi} we obtain immediately that $\mathcal{F}[\varphi_n^{\mathrm{S}}](\xi) = (-\mathrm{i})^n \sqrt{\nu} \Psi_n(\nu\xi)$.
\end{remark}

\section{Malmquist-Takenaka spectral method}\label{sect:MTSM}
In this section we present a novel spectral discretization using MTFs in space combined with time-stepping schemes for the temporal discretization for solving the equation \eqref{eq:Model}.

\subsection{Spatial discretization}
Let $\psi_N(x,t) = \sum_{k=-N}^{N-1} \zeta_k(t) \varphi_k(x)$ be the spectral approximation to the solution of \eqref{eq:Model} and let $U(t)$ denote the coefficient vector of $\psi_N(x,t)$, i.e., $U(t)=(\zeta_{-N}(t),\ldots,\zeta_{N-1}(t))^{T}$. From \eqref{eq:UB} it is easy to see that $\psi_N(x,t)$ automatically satisfies the boundary condition $\lim_{|x|\rightarrow{\infty}}\psi_N(x,t)=0$. Our spectral Galerkin method is to find $\psi_N\in{\mathbb{V}_{N}}(\mathbb{R})$ such that
\begin{align}\label{eq:SGM}
\mathrm{i} (\partial_{t}{\psi_N},\phi) &= \gamma ((-\Delta)^{\alpha/2}\psi_N,\phi) + (\mathcal{T}\psi_N,\phi), \quad \forall \phi\in{\mathbb{V}_{N}(\mathbb{R})},
\end{align}
and $\psi_N(x,0)=\Pi_N\psi_0(x)$. Setting $\phi(x)=\varphi_j(x)$ in \eqref{eq:SGM} with $j=-N,\ldots,N-1$ and using the orthogonality property of MTFs, we obtain that
\begin{align}\label{eq:LS}
U'(t) = -\mathrm{i}\gamma A U(t) - \mathrm{i} \mathcal{N}(U,t),
\end{align}
where $A\in\mathbb{C}^{2N\times{2N}}$, $\mathcal{N}(U,t)\in\mathbb{C}^{2N}$ are defined by
\begin{align}\label{def:AB}
A = \big( (-\Delta)^{\alpha/2}\varphi_{k}, \varphi_{j} \big)_{j,k=-N}^{N-1}, \quad
\mathcal{N}(U,t) = \big(\mathcal{T}\psi_N,\varphi_{j} \big)_{j=-N}^{N-1}.
\end{align}
Furthermore, recalling the Parseval's equality, i.e.,
\begin{equation}
\int_{\mathbb{R}} f(x) \overline{g(x)} \mathrm{d}x = \int_{\mathbb{R}} \mathcal{F}[f](\xi) \overline{\mathcal{F}[g](\xi)} \mathrm{d}\xi, \quad  \forall f,g\in L^2(\mathbb{R}), \nonumber
\end{equation}
it follows that $((-\Delta)^{\alpha/2}\varphi_{k},\varphi_{j})=(|\xi|^{\alpha}\mathcal{F}[\varphi_{k}], \mathcal{F}[\varphi_{j}])$. Hence, the matrix $A$ in \eqref{def:AB} can also be written as
\begin{align}\label{eq:AA}
A = \big( |\xi|^{\alpha}\mathcal{F}[\varphi_{k}], \mathcal{F}[\varphi_{j}] \big)_{j,k=-N}^{N-1}.
\end{align}
Now, we consider the elements of the matrix $A$.
\begin{lemma}\label{thm:MatA}
Let $A$ be the matrix defined in \eqref{def:AB} or \eqref{eq:AA}. Then $A$ is a Hermitian matrix and can be written as a block two-by-two diagonal matrix of the form
\begin{equation}\label{eq:A2}
A = \left[
      \begin{array}{cc}
        PCP &   \\
          & C \\
      \end{array}
    \right],
\end{equation}
where $P\in\mathbb{R}^{N\times N}$ is the permutation matrix which reverses the order of a vector, i.e., $P(x_1,\ldots,x_N)^{T}=(x_N,\ldots,x_1)^T$, and $C\in\mathbb{C}^{N\times N}$ is a Hermitian matrix whose elements are given by
\begin{equation}\label{eq:C}
C_{j,k} = \mathrm{i}^{j-k} \sum_{\ell=0}^{\min\{j,k\}} \frac{(\alpha+1)_{\ell} (-\alpha)_{k-\ell} (-\alpha)_{j-\ell}}{\ell! (k-\ell)! (j-\ell)!}, \quad j,k=0,\ldots,N-1,
\end{equation}
and $(z)_n$ is the Pochhammer symbol defined by $(z)_n=(z)_{n-1}(z+n-1)$ for $n\geq1$ and $(z)_0=1$. In the particular case of $\alpha=1$, then $C$ reduces to a tridiagonal and Hermitian matrix whose elements are given explicitly by
\begin{equation}\label{eq:CI}
C_{j,k} =
\begin{cases}
j(-\mathrm{i}),    \quad  &k=j-1, \\
(2j+1),            \quad  &k=j,   \\
(j+1)\mathrm{i},   \quad  &k=j+1.
\end{cases} \quad  j,k=0,\ldots,N-1.
\end{equation}
\end{lemma}
\begin{proof}
Combining \eqref{eq:AA} with Lemma \ref{lem:FTvarphi} we have
\begin{align}\label{eq:AStepI}
A_{j,k} = \left( |\xi|^{\alpha}\mathcal{F}[\varphi_{k}], \mathcal{F}[\varphi_{j}] \right) &= \mathrm{i}^{j-k} \int_{\mathbb{R}} |\xi|^{\alpha} \Psi_k(\xi) \Psi_j(\xi) \mathrm{d}\xi \nonumber \\
&= \mathrm{i}^{j-k} \begin{cases}
{\displaystyle \int_{\mathbb{R}_{+}} \xi^{\alpha} e^{-\xi} L_k(\xi) L_j(\xi) \mathrm{d}\xi},    \quad  & k,j\geq0, \\[12pt]
{\displaystyle \int_{\mathbb{R}_{+}} \xi^{\alpha} e^{-\xi} L_{-k-1}(\xi) L_{-j-1}(\xi) \mathrm{d}\xi},    \quad  &k,j<0.
\end{cases}
\end{align}
It is easily seen that $A$ is a block two-by-two diagonal matrix and the first block can be derived by reversing the order of rows and columns of the second block. Therefore, we restrict our attention to the second block, which is denoted by $C$. Recalling the connection formula of Laguerre polynomials (see \cite[Equation~(18.18.18)]{Olver2010}), we have
\begin{equation}\label{eq:LagConv}
L_n(\xi) = \sum_{j=0}^n \frac{(-\alpha)_{n-j}}{(n-j)!} L^{(\alpha)}_j(\xi).
\end{equation}
The desired result \eqref{eq:C} follows by combining \eqref{eq:AStepI}, \eqref{eq:LagConv} and the orthogonality property of the Laguerre polynomials $\{L_k^{(\alpha)}(\xi)\}$. In the particular case of $\alpha=1$, recalling the three term recurrence relation of Laguerre polynomials, i.e.,
\begin{equation}
(k+1)L_{k+1}(\xi) = (2k+1-\xi)L_k(\xi) - kL_{k-1}(\xi),
\end{equation}
The desired result \eqref{eq:CI} follows immediately by combining \eqref{eq:AStepI} and the orthogonality of Laguerre polynomials $\{L_k(\xi)\}$. This ends the proof.
\end{proof}

\begin{remark}
We define
\begin{equation}\label{eq:Beta}
\beta_{\ell,n}^{(\alpha)} = \frac{(-\alpha)_{n-\ell}}{(n-\ell)!} \sqrt{\frac{(\alpha+1)_{\ell}}{\ell!}}, \quad \ell=0,\ldots,n,
\end{equation}
where $n=0,1,\ldots$. Note that $\beta_{\ell,n}^{(\alpha)}=0$ for $\alpha=1$ and $n-\ell>1$. Combining \eqref{eq:AStepI}, \eqref{eq:LagConv} and the orthogonality property of the Laguerre polynomials $\{L_k^{(\alpha)}(\xi)\}$, we obtain that
\begin{align}
C_{j,k} &= \mathrm{i}^{j-k} \int_{\mathbb{R}_{+}} \xi^{\alpha} e^{-\xi} L_k(\xi) L_j(\xi) \mathrm{d}\xi = \mathrm{i}^{j-k}\sum_{\ell=0}^{\min\{j,k\}}\beta_{\ell,j}^{(\alpha)}\beta_{\ell,k}^{(\alpha)},
\end{align}
and thus
\begin{align}\label{eq:CII}
C = \left(
      \begin{array}{ccc}
        \mathrm{i}^{0} & \cdots & \mathrm{i}^{1-N} \\
        \vdots & \ddots & \vdots \\
        \mathrm{i}^{N-1} & \cdots & \mathrm{i}^{0} \\
      \end{array}
    \right) \circ & \left[ \left(
                    \begin{array}{ccc}
                      \beta_{0,0}^{(\alpha)} &  &  \\
                      \vdots & \ddots &  \\
                      \beta_{0,N-1}^{(\alpha)} & \cdots & \beta_{N-1,N-1}^{(\alpha)} \\
                    \end{array}
                  \right) \right. \nonumber \\
                  &~~~~~ \times \left. \left(
                    \begin{array}{ccc}
                      \beta_{0,0}^{(\alpha)} &  &  \\
                      \vdots & \ddots &  \\
                      \beta_{0,N-1}^{(\alpha)} & \cdots & \beta_{N-1,N-1}^{(\alpha)} \\
                    \end{array}
                  \right)^T \right],
\end{align}
where $\circ$ denotes the Hadamard product.
\end{remark}

\subsection{Temporal discretization}
In this section we consider time-stepping schemes for the temporal discretization of \eqref{eq:LS}. We divide our discussion into two cases according to $\mathcal{T}$ is a linear or nonlinear operator.

\subsubsection{The linear case}
We restrict our attention to the case $\mathcal{T}\psi = V(x)\psi(x,t)$, where $V(x)$ is a smooth potential. From the definition of $\mathcal{N}(U,t)$ in \eqref{def:AB} we obtain $\mathcal{N}(U,t)=MU(t)$, where $M\in\mathbb{C}^{2N\times2N}$ is defined by $M_{j,k}=(\mathcal{T}\varphi_k,\varphi_j)$ with $j,k=-N,\ldots,N-1$. For $k\in\mathbb{Z}$, we define
\begin{equation}\label{def:mu}
\mu_k = \frac{\mathrm{i}^{-k}}{2\pi} \int_{-\pi}^{\pi} V\left(\frac{1}{2}\tan{\frac{\theta}{2}}\right) e^{-\mathrm{i}k\theta} \mathrm{d}\theta.
\end{equation}
It can be verified by direct calculation that $M_{j,k}=\mu_{j-k}$, and thus
\begin{equation}\label{def:M}
M = \left(
    \begin{array}{cccc}
       \mu_0 & \mu_{-1} & \cdots & \mu_{1-2N} \\
       \mu_1 & \mu_0    & \cdots & \mu_{2-2N} \\
       \vdots & \vdots  & \ddots & \vdots     \\
       \mu_{2N-1} & \mu_{2N-2} & \cdots  & \mu_0
    \end{array}
    \right).
\end{equation}
Clearly, we see that $M$ is a Toeplitz and Hermitian matrix. Furthermore, observe that the integrand on the right-hand of \eqref{def:mu} is periodic in $\theta$ with period $2\pi$, the elements of $M$ (i.e., $\{\mu_k\}_{k=1-2N}^{2N-1}$) can be computed rapidly with the FFT in $\mathcal{O}(N\log N)$ operations.

Let us now turn to the numerical solution of \eqref{eq:LS}. In view of $\mathcal{N}(U,t)=MU(t)$, we obtain the following ODE system
\begin{equation}\label{eq:DEU}
U'(t) = -\mathrm{i}(\gamma{A}+M)U(t).
\end{equation}
The exact solution of \eqref{eq:DEU} is $U(t) = \exp(-\mathrm{i}(\gamma{A}+M)t)U(0)$ and $U(0)$ can be computed from the Malmquist-Takenaka coefficients of $\psi_0(x)$ by the FFT. Note that the exact solution requires the computation of the matrix exponential $\exp(-\mathrm{i}(\gamma{A}+M)t)$, which is generally not a good idea to compute it directly. Indeed, as we will see in Figures \ref{fig:C} and \ref{fig:decaypro}, the decay rates of the magnitudes of the elements of $A$ and $M$ (or, equivalently, $C$ and $M$) are quite different. To avoid this issue, we consider the use of splitting method to solve \eqref{eq:DEU}. Specifically, let $t_k=k\tau$ denote the time grid points, where $\tau>0$ is the time step size, and let $U_k$ denote the approximation to the exact value $U(t_k)$ and $U_0=U(0)$. From $t_{n-1}$ to $t_n$, the splitting method reads
\begin{equation}\label{eq:ExpSplit}
U_{n} = \underbrace{\left[\prod_{j=1}^{m} \exp(-\mathrm{i}a_j\gamma\tau A) \exp(-\mathrm{i}b_j\tau M)\right]}_{ :=~ S(\tau)} U_{n-1}, \quad n\geq1,
\end{equation}
where $a_j$ and $b_j$ are some suitably chosen coefficients to ensure that the method achieves some order $p$, i.e., $S(\tau)=\exp(-\mathrm{i}(\gamma{A}+M)\tau)+\mathcal{O}(\tau^{p+1})$. We list below three symmetric splitting methods which achieve orders two, four and six, respectively:
\begin{itemize}
\item SM1:
\begin{equation}\label{eq:SM1}
a_1= a_2 = \frac{1}{2},~~~ b_1=1,~ b_2=0.
\end{equation}
This method is known as the Strang splitting \cite{strang1968}.
\item SM2:
\begin{equation}\label{eq:SM2}
a_1 = a_4 = \frac{\kappa_1}{2},~ a_2 = a_3 = \frac{\kappa_0+\kappa_1}{2}, ~~~ b_1=b_3=\kappa_1,~ b_2=\kappa_0,~b_4=0,
\end{equation}
where $\kappa_0=-2^{1/3}/(2-2^{1/3})$ and $\kappa_1=1/(2-2^{1/3})$.

\item SM3:
\begin{align}\label{eq:SM3}
a_1 &= a_8 = \frac{w_3}{2}, ~a_2=a_7=\frac{w_2+w_3}{2},~ a_3 = a_6 = \frac{w_1+w_2}{2},~a_4=a_5 = \frac{w_0+w_1}{2}, \notag\\
b_1 &= b_7 = w_3,~b_2=b_6=w_2,~b_3 = b_5 = w_1,~b_4=w_0,~b_8=0,
\end{align}
where $w_1=-1.17767998417887$, $w_2=0.235573213359$, $w_3=0.784513610477$ and $w_0 = 1-2(w_1+w_2+w_3)$.
\end{itemize}
For the derivation of the coefficients of these splitting methods, we refer to \cite{Yoshida1990} for more details.

\begin{figure}
\centering
\includegraphics[height=5cm,width=7cm]{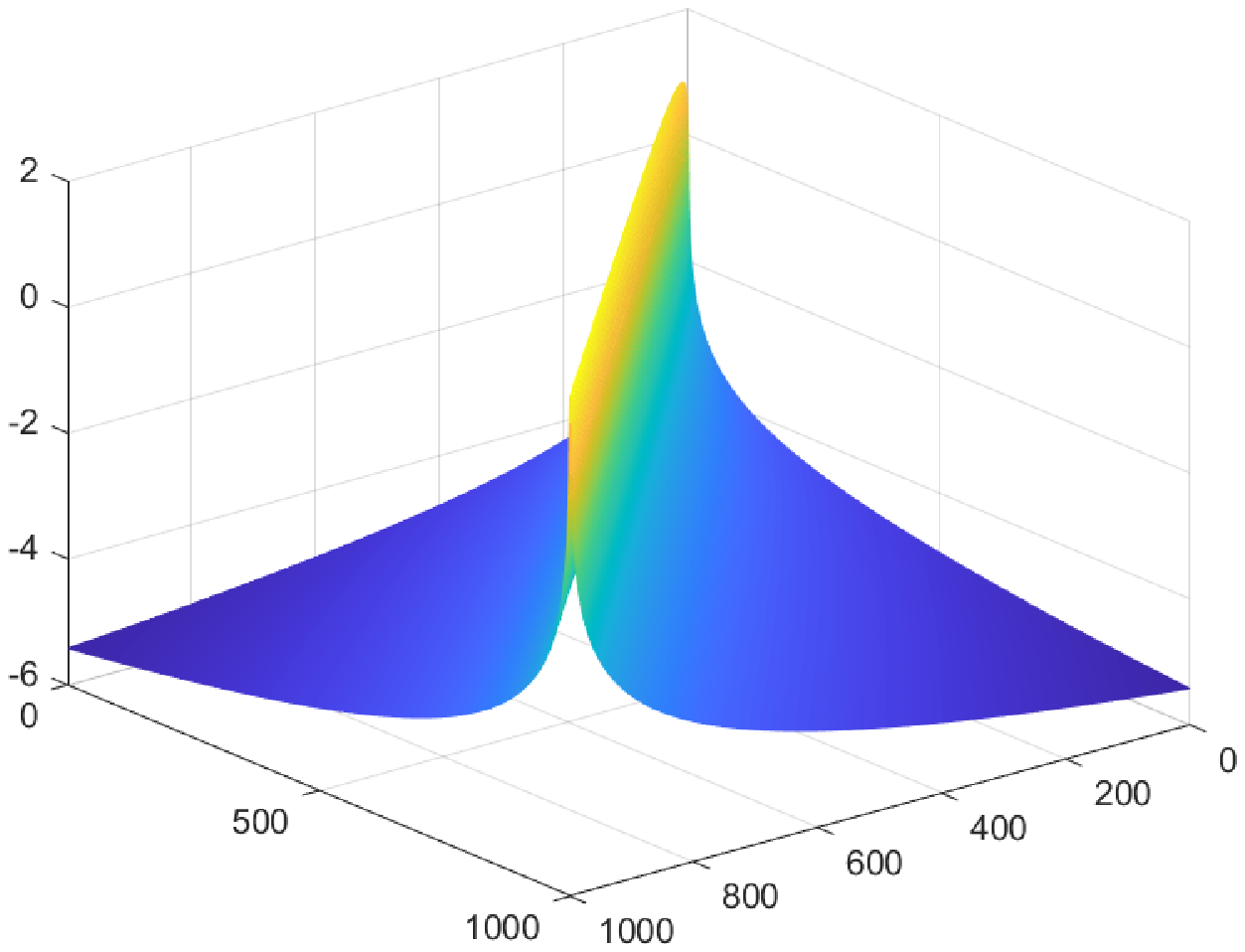}
\includegraphics[height=5cm,width=7cm]{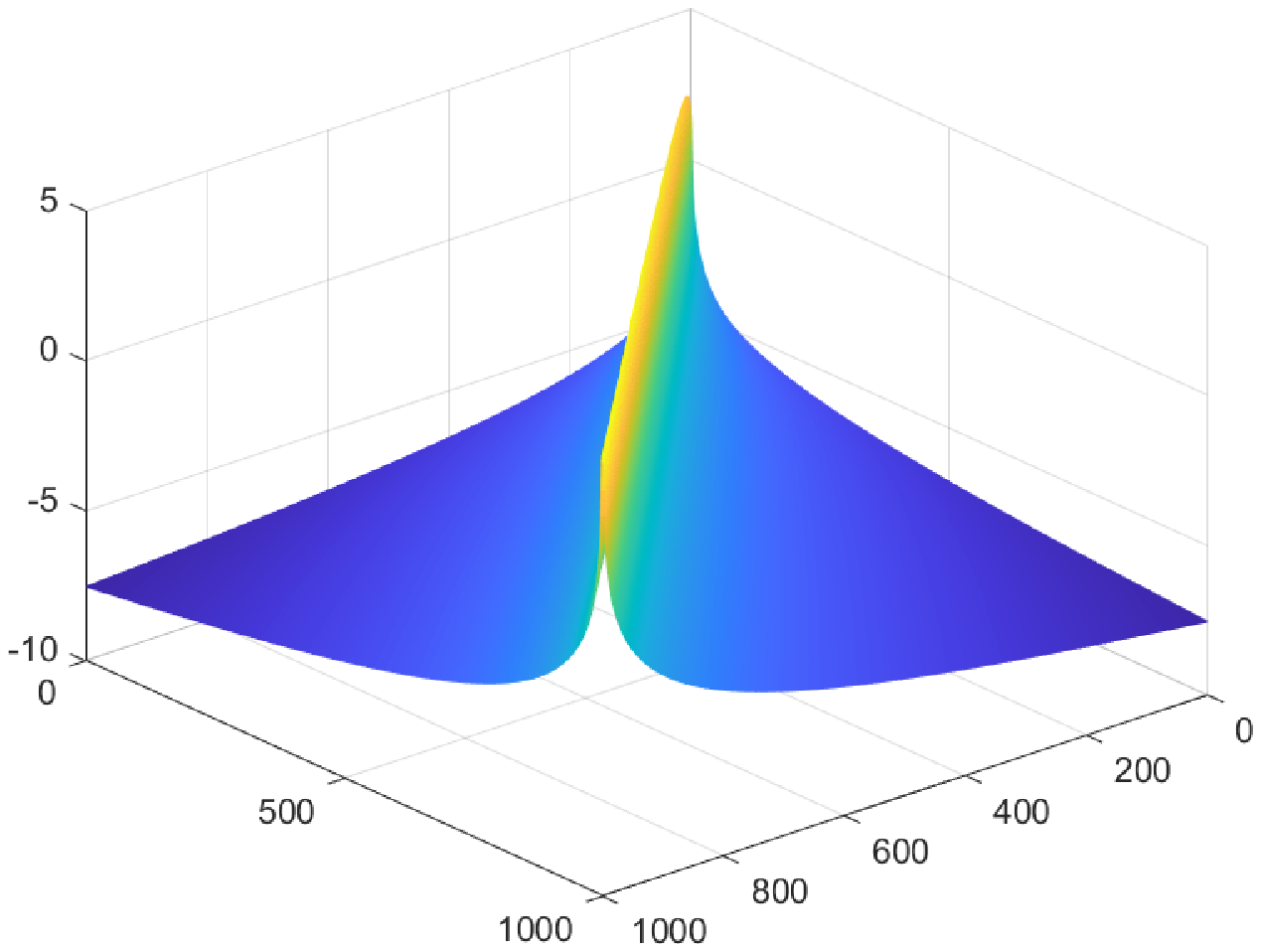}
\caption{The logarithm of the magnitude of the elements of $C$ for $\alpha=0.6$ (left) and $\alpha=1.4$ (right).}
\label{fig:C}
\end{figure}

\begin{remark}
Since both $A$ and $M$ are Hermitian matrices, those matrix exponentials in \eqref{eq:ExpSplit} are all unitary and thus unitary evolution and unconditional stability of the method are guaranteed.
\end{remark}

When implementing splitting methods, it is necessary to evaluate matrix exponential of the forms
$\exp(-\mathrm{i}\lambda{A})$ and $\exp(-\mathrm{i}\lambda M)$, where $\lambda>0$ is a constant. Moreover, from \eqref{eq:A2} we obtain that
\begin{equation}
\exp\left(-\mathrm{i}\lambda{A}\right) = \left[
      \begin{array}{cc}
        P \exp\left(-\mathrm{i}\lambda{C}\right) P &   \\
          & \exp\left(-\mathrm{i}\lambda{C}\right)  \\
      \end{array}
    \right],
\end{equation}
and thus the computation of $\exp(-\mathrm{i}\lambda{A})$ can be reduced to the computation of $\exp(-\mathrm{i}\lambda{C})$. We now consider the structure of the matrices $C$ and $M$. In the case $\alpha\neq1$, from \eqref{eq:Beta} we deduce that $\beta_{\ell,n}^{(\alpha)}=\mathcal{O}(n^{-\alpha-1})$ for fixed $\ell$ and $n\gg1$ and therefore the magnitude of the elements of $C$ decays algebraically away from the diagonal; see Figure \ref{fig:C}. In the case $\alpha=1$, from Lemma \ref{thm:MatA} we know that $C$ is a Hermitian and tridiagonal matrix. Moreover, from \eqref{def:M} we know that $M$ is also a Hermitian matrix and from \eqref{def:mu} it is easily seen that ${\mu_k}$ decays rapidly whenever $V\in{H^{s}(\mathbb{R})}$ for sufficiently large $s\in\mathbb{N}$, and thus we can expect that $M$ is near a banded matrix; see Figure \ref{fig:decaypro}. Now we turn to the computation of $\exp(-\mathrm{i}\lambda{C})$ and $\exp(-\mathrm{i}\lambda{M})$. For this, we use the recently developed algorithm in \cite{Bader2022}, which is based on Chebyshev approximation of complex exponentials. We only consider the computation of $\exp(-\mathrm{i}\lambda{C})$ since the computation of $\exp(-\mathrm{i}\lambda{M})$ is similar. Specifically, if the eigenvalues of $C$ are contained in the interval $[\zeta,\eta]\subset\mathbb{R}$ and $\lambda(\eta-\zeta)/2\leq2.212$, the algorithm in \cite{Bader2022} reads
\begin{equation}\label{eq:ChebH}
\exp(-\mathrm{i}\lambda{C}) \approx \exp\left(-\mathrm{i}\lambda\frac{\eta+\zeta}{2}\right) \left[ c_0 I_N + 2\sum_{k=1}^{m} c_k T_k\left( \frac{2}{\eta-\zeta}\left(C - \frac{\zeta+\eta}{2} I_N  \right) \right) \right],
\end{equation}
where $I_{N}$ is the identity matrix of order $N$ and $c_k = (-\mathrm{i})^kJ_k(\lambda(\eta-\zeta)/2)$ and $J_k(x)$ is the Bessel function of the first kind of order $k$. When choosing $m=18$, the approximation error in \eqref{eq:ChebH} will be less than the machine precision (i.e., $2^{-53}\approx1.11\times10^{-16}$) and the calculation of the right-hand side of \eqref{eq:ChebH} can be achieved with only five matrix-matrix products. Otherwise, if $\lambda(\eta-\zeta)/2>2.212$, then the scaling and squaring technique, i.e., $\exp(-\mathrm{i}\lambda{C})=(\exp(-\mathrm{i}\lambda{C}/2^s))^{2^s}$ for some $s\in\mathbb{N}$, should be used such that the exponential $\exp(-\mathrm{i}\lambda{C}/2^s)$ can be evaluated by \eqref{eq:ChebH}. Note that the values of $\zeta$ and $\eta$ have to be specified before embarking on the algorithm. If they cannot be specified in advance, then the algorithm will simply take $\eta=-\zeta=\|C\|_1$.

\begin{figure}
\centering
\includegraphics[height=5.0cm,width=7cm]{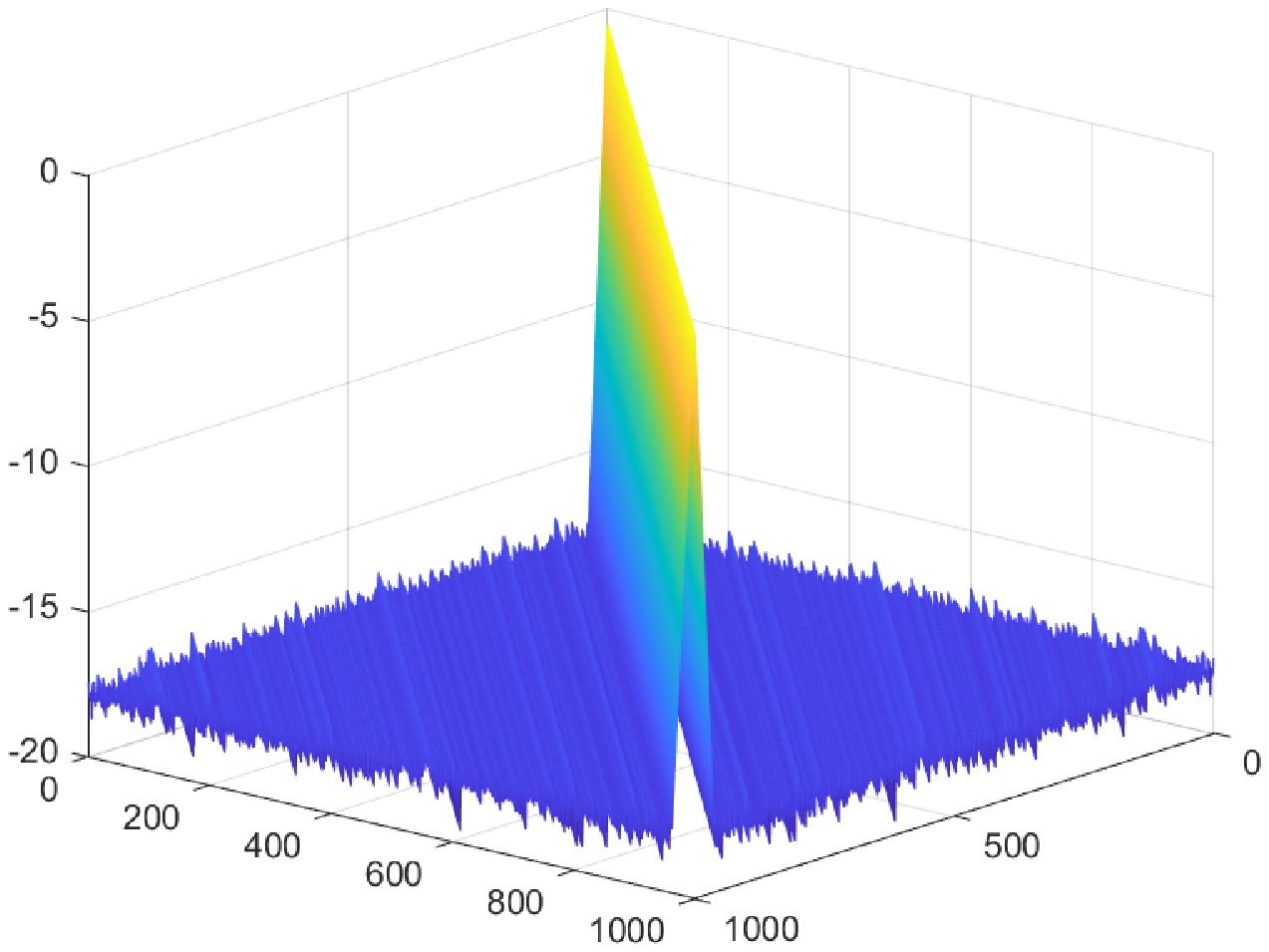}
\includegraphics[height=5.0cm,width=7cm]{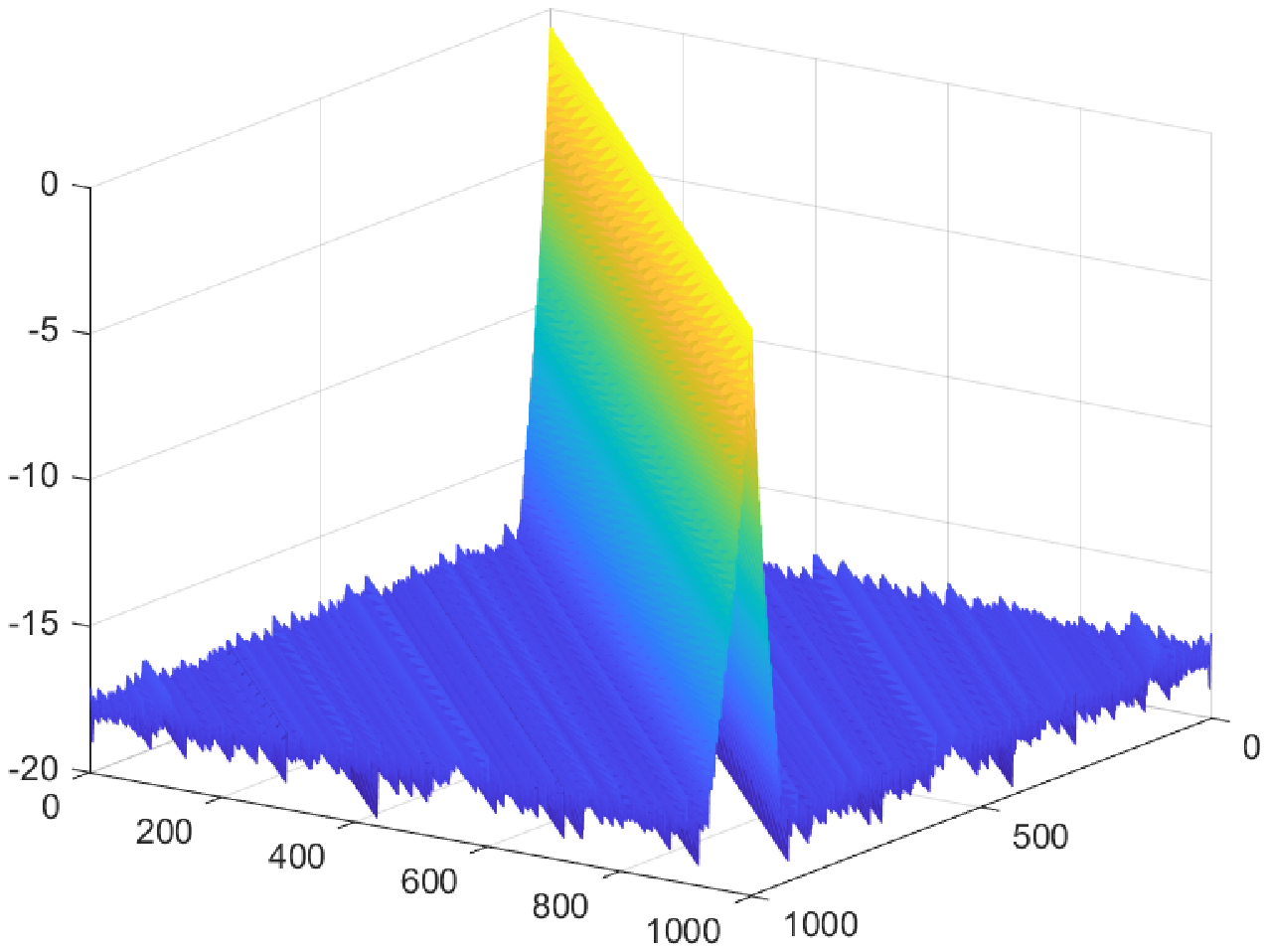}
\caption{The logarithm of the magnitude of the elements of $M$ for $V(x)=1/(1+x^2)$ (left) and $V(x)=\exp(-x^2)$ (right).}
\label{fig:decaypro}
\end{figure}

\begin{remark}
It is possible to improve the efficiency of the algorithm \eqref{eq:ChebH} by taking the structure of $C$ and $M$ into account. For example, we can simply set the elements of the matrices involved to zero whenever their magnitude is less than the machine precision.
\end{remark}

\subsection{The nonlinear case}
In the case where $\mathcal{T}$ is a nonlinear operator, using the variation-of-constant formula to \eqref{eq:LS}, we have
\begin{equation}\label{eq:Nonlinear}
U(t_{n+1}) = e^{-\mathrm{i}\gamma\tau{A}} U(t_n) - \mathrm{i}\tau \int_{0}^{1} e^{-\mathrm{i}\gamma\tau{A}(1-y)}  \mathcal{N}(U(t_{n}+y\tau),t_{n}+y\tau) \mathrm{d}y.
\end{equation}
To approximate \eqref{eq:Nonlinear}, the fourth-order exponential time differencing Runge-Kutta (ETDRK4) method and its various modifications are preferable (see, e.g., \cite{Bhatt2016,Cox2002,Kassam2005,Krogstad2005}). Here we utilize the Krogstad-P22 scheme developed in \cite{Bhatt2016}. More specifically, let $L=\mathrm{i}\gamma{A}$ and let $R_{2,2}(\tau{L}) = (12I-6\tau{L}+\tau^2{L}^2)(12{I}+6\tau{L}+\tau^2{L}^2)^{-1}$, where we have omitted the subscript $2N$ on the identity matrix $I$ for notational simplicity. Moreover, we define
\begin{align}
P_1(\tau{L}) &= 12\tau(12{I}+6\tau{L}+\tau^2{L}^2)^{-1} \notag , \\
P_2(\tau{L}) &= \tau(6{I}+\tau{L})(12{I}+6\tau{L}+\tau^2{L}^2)^{-1},  \notag \\
P_3(\tau{L}) &= 2\tau(4{I}+\tau{L})(12{I}+6\tau{L}+\tau^2{L}^2)^{-1}. \notag
\end{align}
Then, the Krogstad-P22 scheme reads
\begin{align}\label{eq:ETDRKPade}
U_{n+1} &= R_{2,2}(\tau{L})U_n - \mathrm{i} {P}_1(\tau{L}) \mathcal{N}(U_n,t_n) - \mathrm{i} {P}_2(\tau{L}) \bigg[-3\mathcal{N}(U_n,t_n) + 2\mathcal{N}\left(a_n,t_n+\frac{\tau}{2}\right)  \nonumber \\
&~~~  + 2\mathcal{N}\left(b_n,t_n+\frac{\tau}{2}\right) - \mathcal{N}(c_n,t_n+\tau) \bigg] - \mathrm{i} P_3(\tau{L}) \bigg[\mathcal{N}(U_n,t_n)   \\
&~~~   - \mathcal{N}\left(a_n,t_n+\frac{\tau}{2}\right) - \mathcal{N}\left(b_n,t_n+\frac{\tau}{2}\right) + \mathcal{N}(c_n,t_n+\tau) \bigg], \notag
\end{align}
where
\begin{align}
a_n &= \tilde{R}_{2,2}(\tau{L})U_n - \mathrm{i}\tilde{P}_1(\tau{L}) \mathcal{N}(U_n,t_n)        \notag, \\
b_n &= \tilde{R}_{2,2}(\tau{L})U_n - \mathrm{i}\tilde{P}_1(\tau{L}) \mathcal{N}(U_n,t_n) - \mathrm{i}\tilde{P}_2(\tau{L}) \left[\mathcal{N}\left(a_n,t_n+\frac{\tau}{2}\right) - \mathcal{N}(U_n,t_n) \right],  \notag \\
c_n &= R_{2,2}(\tau{L})U_n -\mathrm{i}P_1(\tau{L}) \mathcal{N}(U_n,t_n) - 2\mathrm{i}P_2(\tau{L}) \left[ \mathcal{N}\left(b_n,t_n+\frac{\tau}{2}\right) - \mathcal{N}(U_n,t_n) \right],  \notag
\end{align}
and
\begin{align}
\tilde{R}_{2,2}(\tau{L}) &= (48{I}-12\tau{L}+\tau^2{L}^2) (48{I}+12\tau{L}+\tau^2{L}^2)^{-1}, \notag  \\
\tilde{P}_1(\tau{L}) &= 24\tau(48{I}+12\tau{L}+\tau^2{L}^2)^{-1},  \notag  \\
\tilde{P}_2(\tau{L}) &= 2\tau(12{I}+\tau{L})(48{I}+12\tau{L}+\tau^2{L}^2)^{-1}. \notag
\end{align}

\begin{remark}
The main drawback of the ETDRK4 and ETDRK4-B schemes \cite{Cox2002,Krogstad2005} is the computation of the following matrix functions
\begin{equation}
\varphi_0(\tau{L}) = \exp(-\tau{L}), \quad  \varphi_k(\tau{L}) = (-\tau L)^{-k} \left( \varphi_0(\tau{L}) - \sum_{j=0}^{k-1} \frac{(-\tau{L})^j}{j!} \right), \quad k=1,2,3, \notag
\end{equation}
which suffers from cancellation errors, especially when the eigenvalues of $L$ are very close to zero. Kassam and Trefethen in \cite{Kassam2005} proposed to evaluate these functions by using complex contour integrals. However, the choice of the contour is problem dependent. The Krogstad-P22 scheme \cite{Bhatt2016} is developed by using Pad\'{e} approximations to the above functions, which avoids direct computation of the higher powers of matrix inverse. Moreover, as observed in \cite{Bhatt2016}, the factors ${L}^{-1}$ and ${L}^{-3}$ that appear in ETDRK4 and  ETDRK4-B schemes cancel out in the Krogstad-P22 scheme.
\end{remark}

The stability of the Krogstad-P22 scheme can be analyzed by using a similar argument as for the ETDRK4 scheme in \cite{Cox2002}. For the nonlinear and autonomous ODE of the form $U_t = -\lambda{U} - \mathrm{i} \mathcal{N}(U)$. Suppose that there exists a fixed point $U_0$ such that $\lambda{U_0}+\mathrm{i}\mathcal{N}(U_0)=0$. Linearizing about this fixed point yields
\begin{equation}\label{nonlinearODE2}
U_t = -\lambda{U}-\mathrm{i}cU,
\end{equation}
where $U$ is now the perturbation to $U_0$ and $c=\mathcal{N}'(U_0)$. Applying the Krogstad-P22 scheme \eqref{eq:ETDRKPade} to the linearized equation \eqref{nonlinearODE2} and setting $r=U_{n+1}/U_n$, $x=-\mathrm{i}c\tau$ and $y=-\lambda\tau$, we then obtain
\begin{equation}\label{eq:afactor}
r(x,y) = c_0+c_1x+c_2x^2+c_3x^3+c_4x^4,
\end{equation}
where
\begin{align*}
c_0 & = \frac{12+6y+y^2}{12-6y+y^2}, \\
c_1 & = \frac{144y^3 - 432y^2 - 1728y + 6912}{{(12-6y+y^2)}^2{(48-12y+y^2)}}, \\
c_2 & = \frac{36y^5 - 648y^4 + 4032y^3 + 3456y^2 - 82944y + 165888}{{(12-6y+y^2)}^2{(48-12y+y^2)}^2}, \\
c_3 & = \frac{- 48y^4 + 768y^3 + 576y^2 - 27648y + 55296}{{(12-6y+y^2)}^2{(48-12y+y^2)}^2}, \\
c_4 & = \frac{- 96y^3 + 1920y^2 - 10368y + 13824}{{(12-6y+y^2)}^2{(48-12y+y^2)}^2}.
\end{align*}
Finally, the stability region of the Krogstad-P22 scheme can be obtained by requiring $|r(x,y)|\leq1$. Since all eigenvalues of ${L}$ are pure imaginary, it is therefore enough to consider the case where $\lambda$ is a pure imaginary number. In Figure \ref{stabilityregion} we plot the stability regions for several pure imaginary numbers of $y$ for \eqref{eq:LS}. We observe that each of the stability region includes an interval of the imaginary axis and its length increases as $|y|$ increases. This gives an indication of stability of the Krogstad-P22 scheme.

\begin{figure}
\centering
\includegraphics[height=7.0cm,width=9cm]{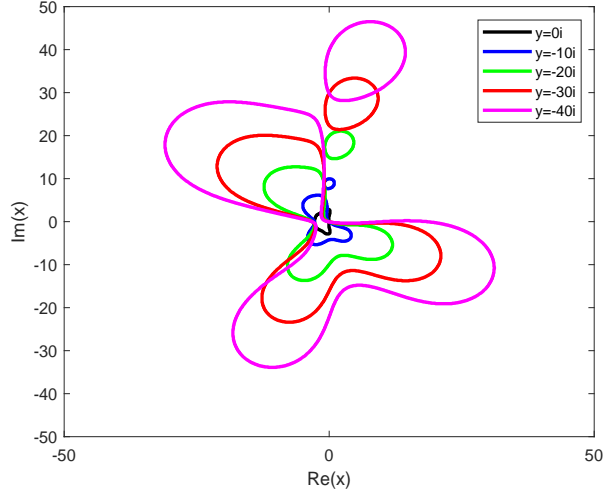}
\caption{Stability regions of the Krogstad-P22 scheme for several values of $y$.}
\label{stabilityregion}
\end{figure}

\section{Numerical examples}\label{sect:Exam}
In this section, we present two examples to show the performance of the proposed method.

\vspace{.2cm}

\noindent{\bf Example 1}.
Consider the following linear fractional Schr\"odinger equation
\begin{equation}\label{eq:FSEpotential}
\mathrm{i}\partial_t\psi(x,t) = \gamma(-\Delta)^{\alpha/2}\psi(x,t) + V(x)\psi(x,t),
\end{equation}
where we take $\gamma=1/2$ and $V(x)=1/(1+x^2)$. We first compare the performance of spatial discretizations using MTFs and MCFs, respectively. To avoid the influence of the error due to temporal discretization, we consider a pure version of spectral methods, that is, we evaluate $U(t)$ by the exact formula $U(t)=\exp(-\mathrm{i}(\gamma{A}+M)t)U(0)$ and compute the involved matrix exponential by the \texttt{expm} function in {\sc Matlab}. Since the exact solution of \eqref{eq:FSEpotential} is not known, we define a reference solution which is computed by the MTF spectral Galerkin method with $N=500$ (note that the number of terms of this spectral method is $2N$). Moreover, we take the scaling parameter $\nu=4$ for both methods.

In the first row of Figure \ref{fsevsech} we plot the maximum error of both spectral methods at time $t=1$ for the initial data $\psi_{0}(x)=\mathrm{sech}(x)$. We see that MTF spectral discretization converges much faster than its MCF counterpart in the case of $\alpha=1$ and both spectral discretizations converge almost at the same rate otherwise. Indeed, in the case of $\alpha=1$, MTF spectral discretization scheme converges at exponential rates, while MCF spectral discretization scheme converges only at algebraic rates. In the second row of Figure \ref{fsevsech} we plot the maximum error at time $t=1$ for the initial data $\psi_{0}(x)=(\mathrm{i}x+10)/(x^2+4)$. We observe that MTF spectral discretization converges always faster than its MCF counterpart. Moreover, similar to the previous case, our MTF spectral discretization converges at exponential rates in the particular case of $\alpha=1$ while its MCF counterpart converges only at algebraic rates.

\begin{figure}
\centering
\includegraphics[height=5.5cm,width=7cm]{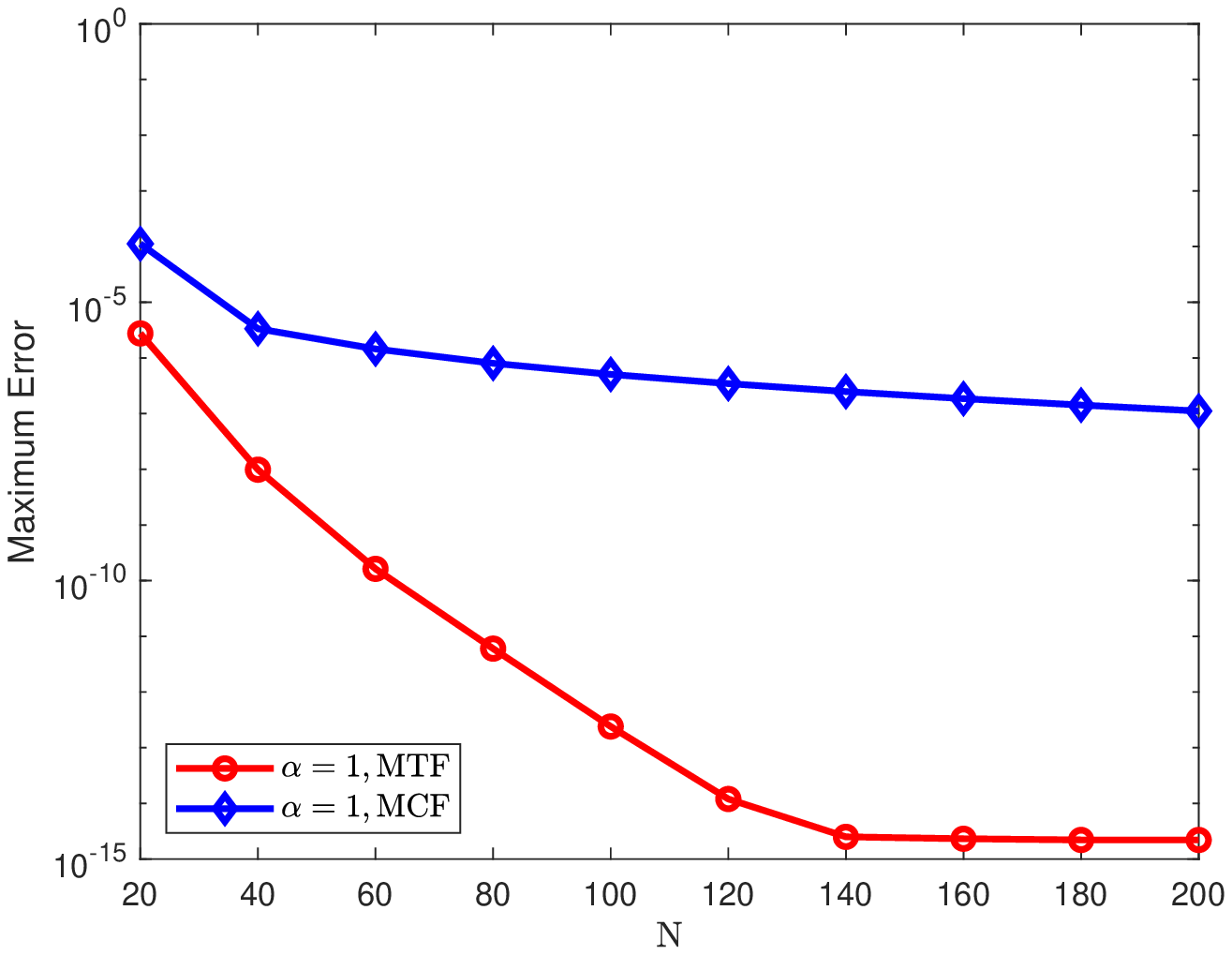}
\includegraphics[height=5.5cm,width=7cm]{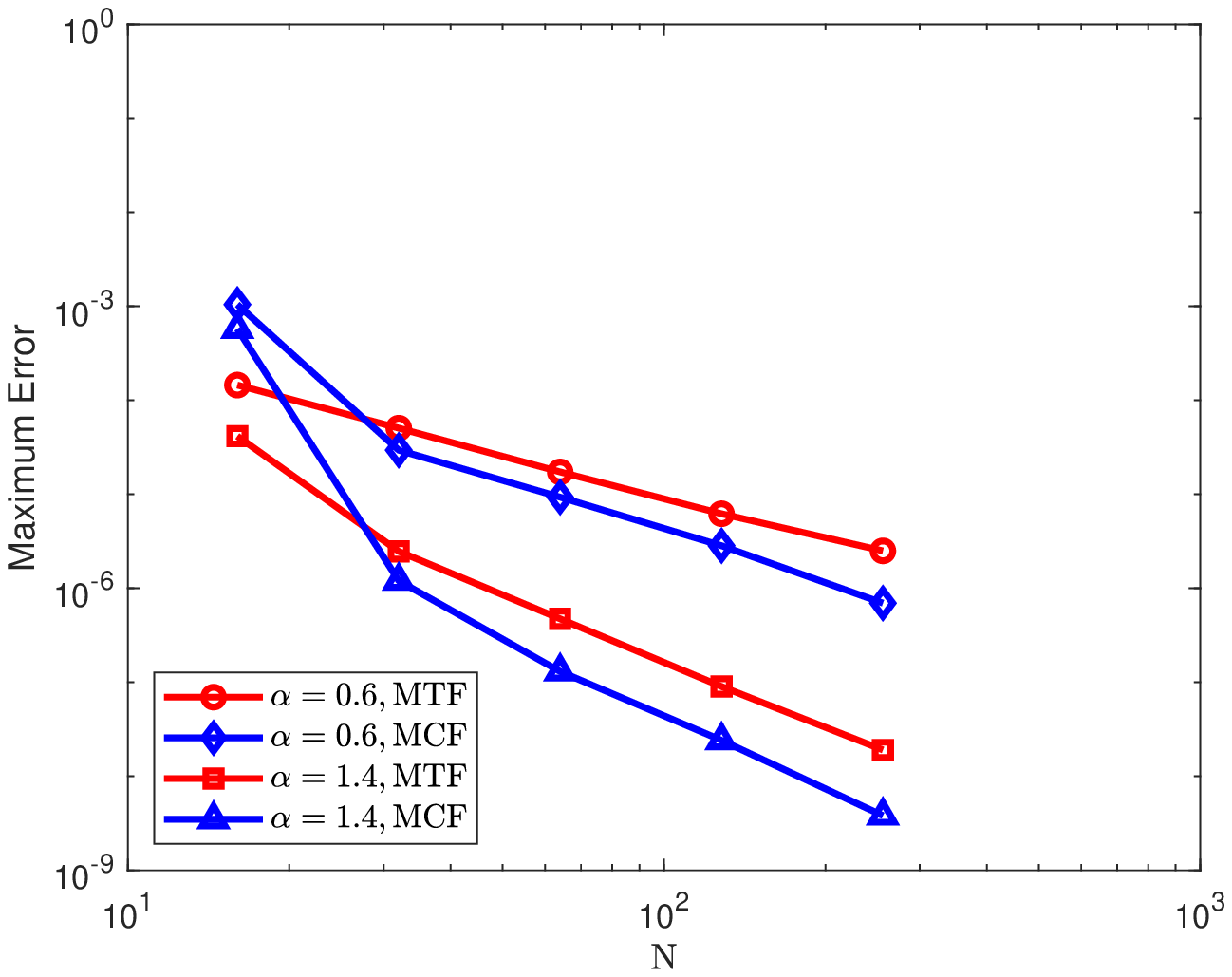}\\
\includegraphics[height=5.5cm,width=7cm]{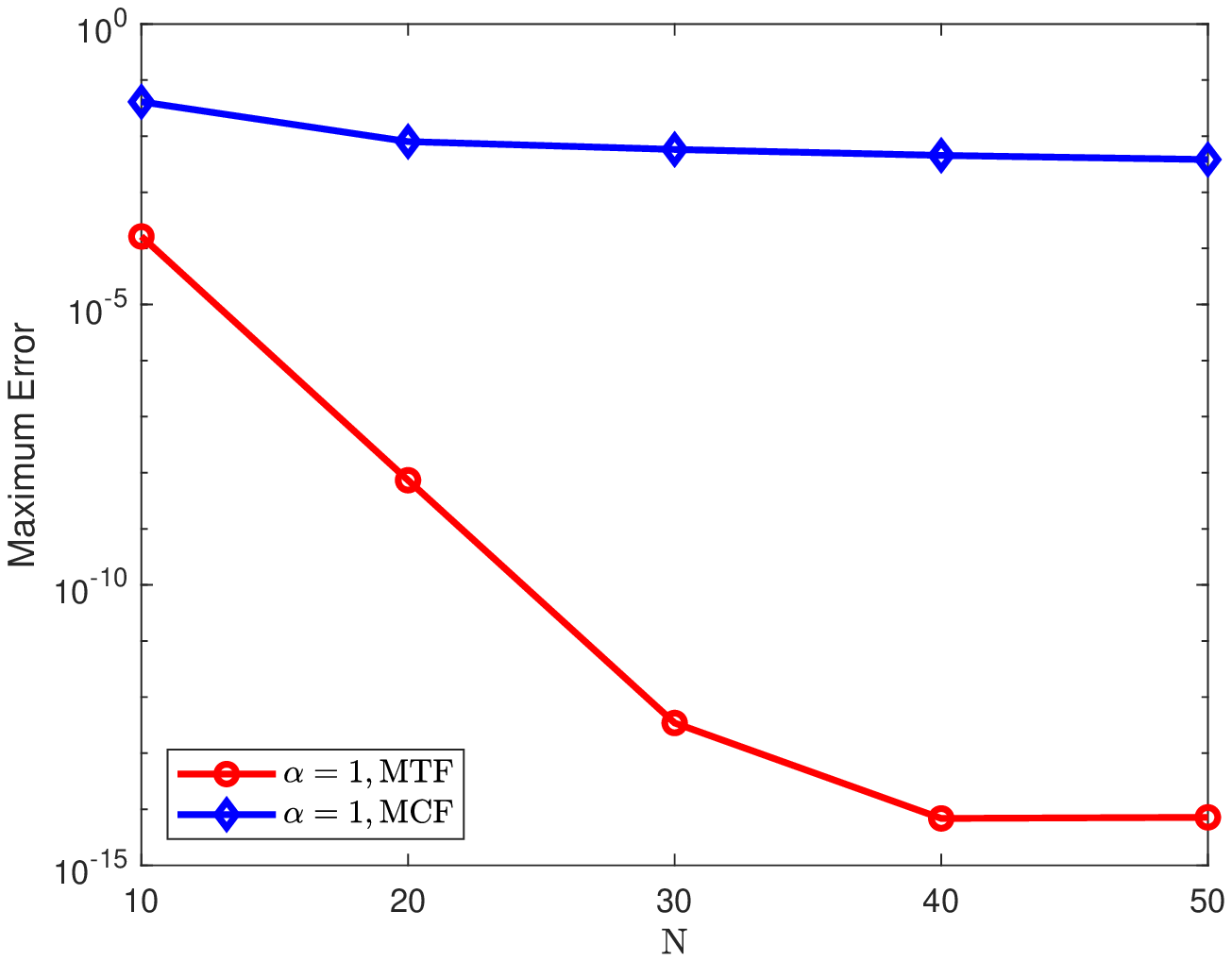}
\includegraphics[height=5.5cm,width=7cm]{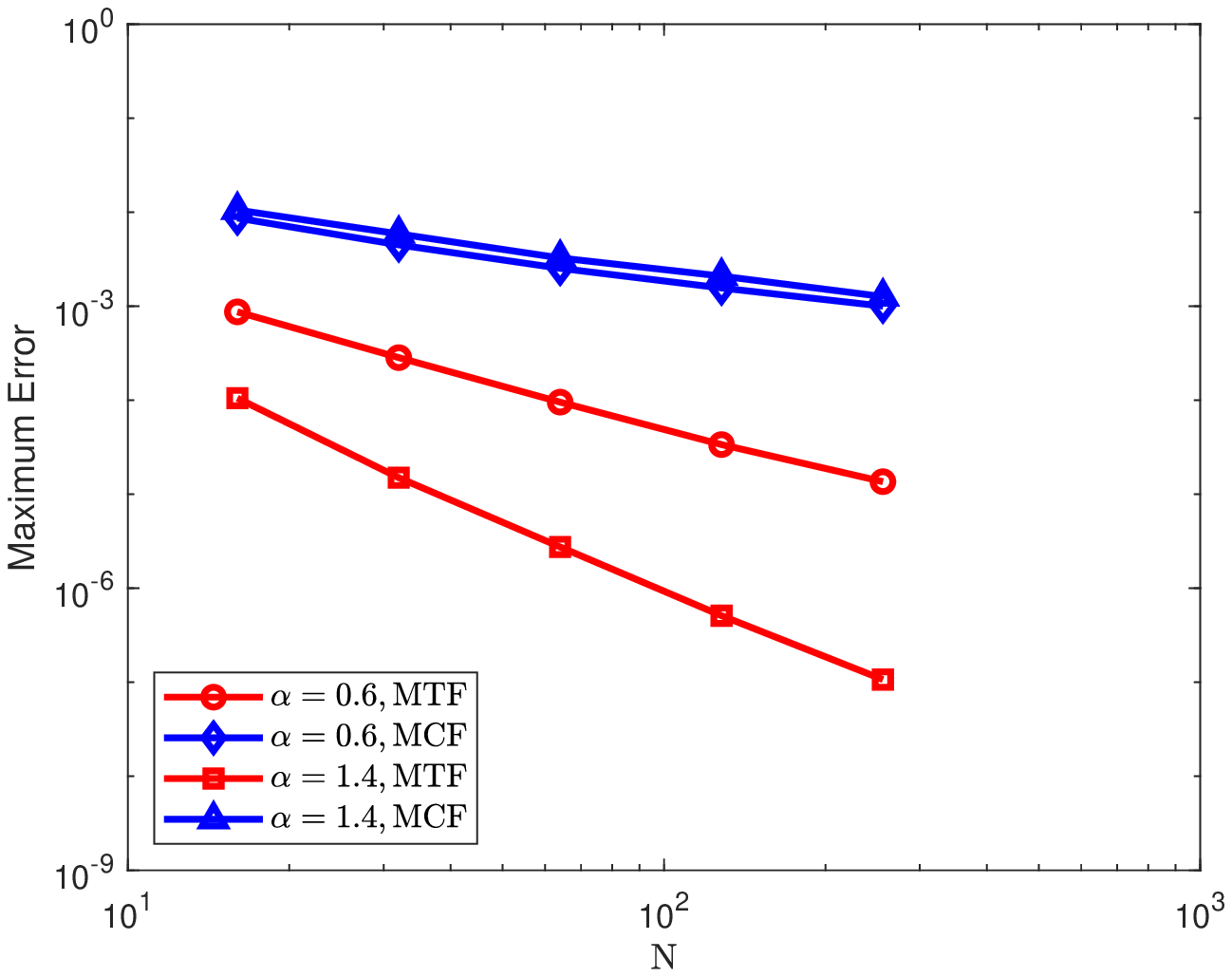}
\caption{Maximum errors of MTF and MCF spectral discretizations for $\alpha=1$ (left), $\alpha=0.6,1.4$ (right). Here $\psi_{0}(x)=\mathrm{sech}(x)$ (top row) and $\psi_{0}(x)=(\mathrm{i}x+10)/(x^2+4)$ (bottom row).}\label{fsevsech}
\end{figure}

\begin{figure}
\centering
\includegraphics[height=4.5cm,width=4.5cm]{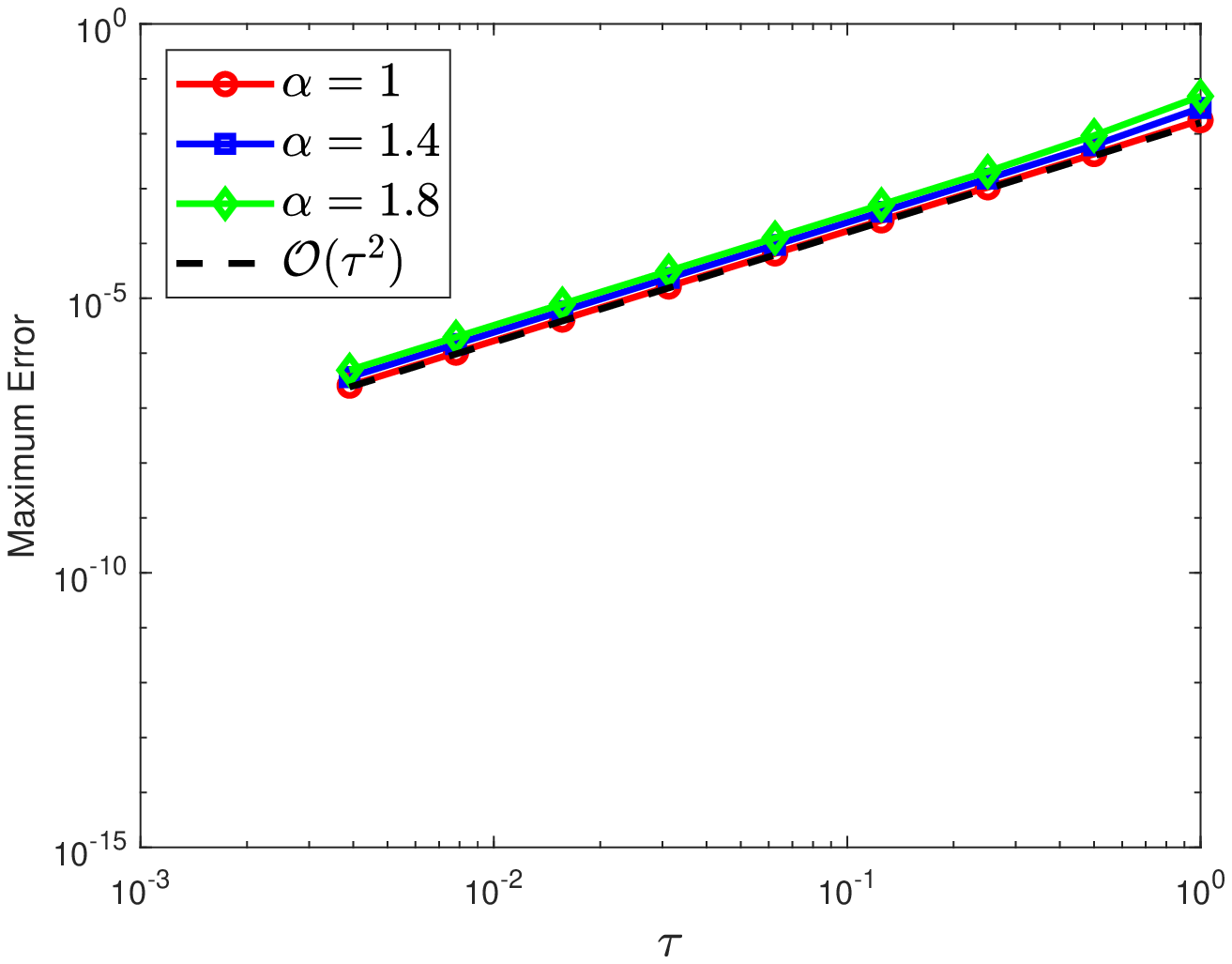}
\includegraphics[height=4.5cm,width=4.5cm]{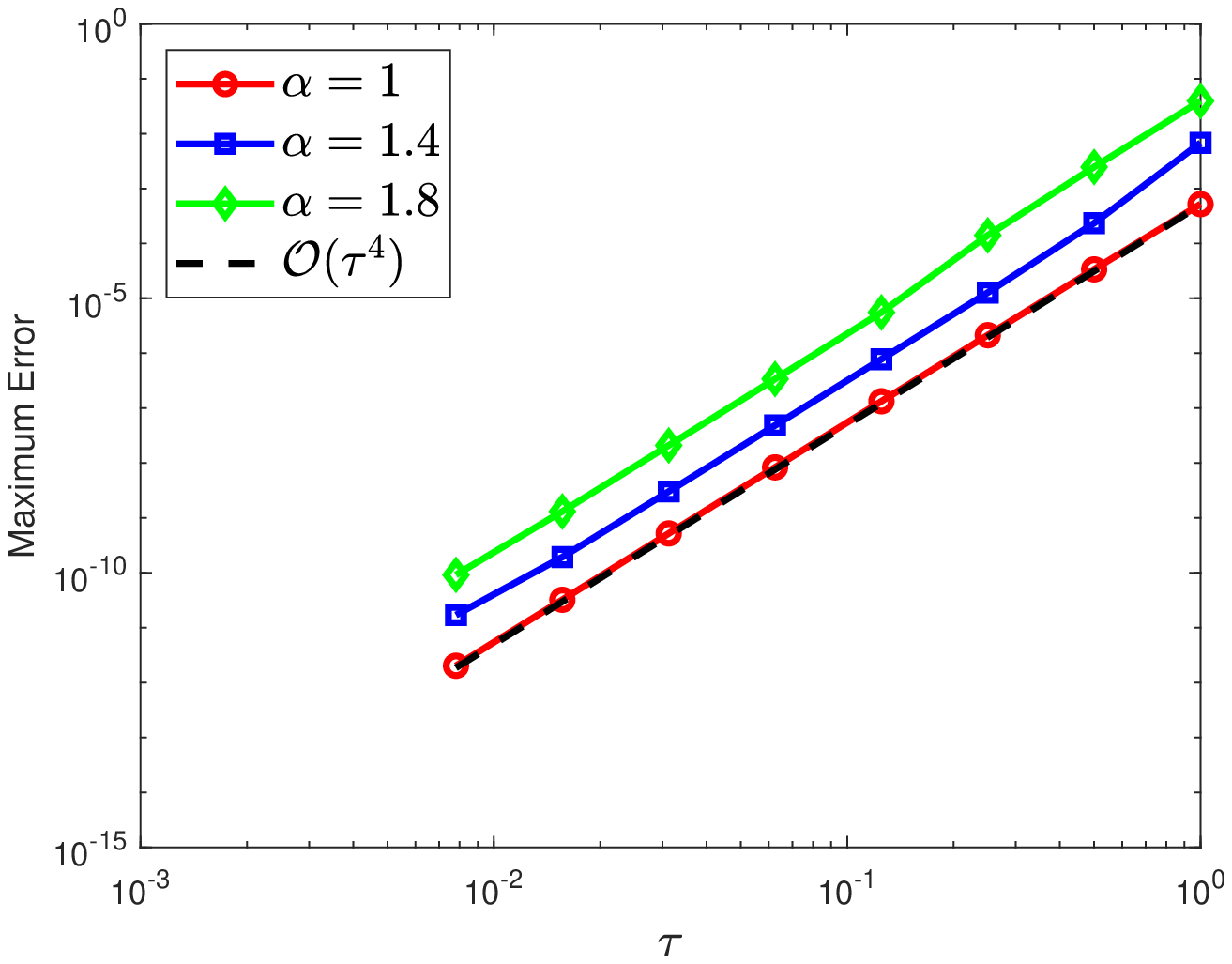}
\includegraphics[height=4.5cm,width=4.5cm]{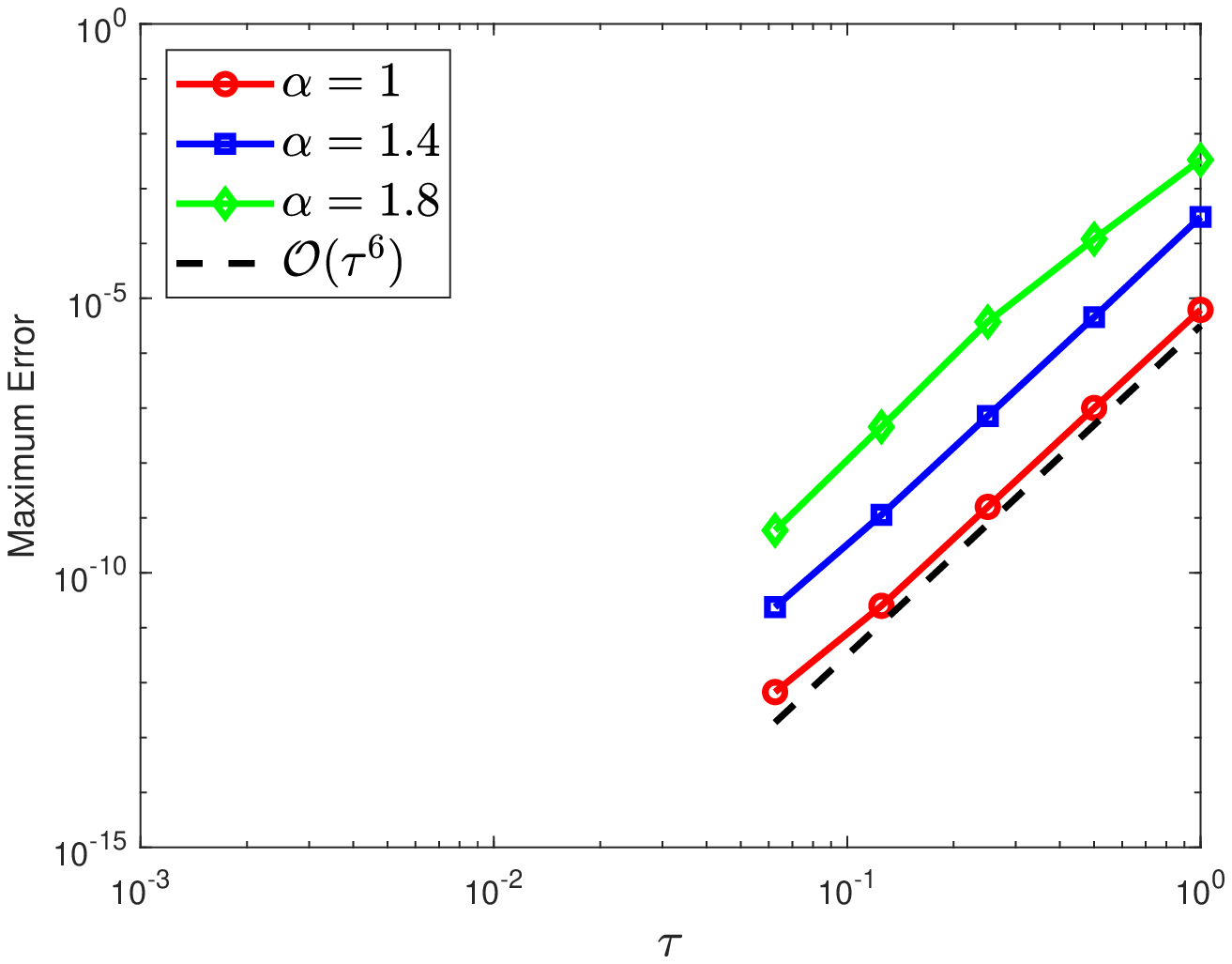}
\caption{Temporal orders of splitting methods SM1 (left), SM2 (middle) and SM3 (right) coupled with MTF spectral Galerkin method for \eqref{eq:FSEpotential}.}
\label{timeorder}
\end{figure}

Next, we consider the temporal order of convergence of the MTF spectral Galerkin method coupled with the splitting methods. In this case, the reference solution is computed with the MTF spectral Galerkin method coupled with the splitting method \eqref{eq:SM3} with $\tau=2^{-11}$ and $N=500$. We consider the maximum error of the numerical solution at time $t=1$ for the initial data $\psi_{0}(x)=1/(1+x+x^2)$. In Figure \ref{timeorder} we plot the maximum error of MTF spectral Galerkin method coupled with the splitting methods \eqref{eq:SM1}, \eqref{eq:SM2} and \eqref{eq:SM3} as a function of the time step size $\tau$. Clearly, we see that the classical order of each splitting scheme is retained.

Finally, we consider the evolution of $|\psi(x,t)|$ in \eqref{eq:FSEpotential} with a double-barrier potential of the form $V(x)=100(\exp(-(x-10)^2)+\exp(-(x+10)^2))$. The initial data is taken as $\psi_{0}(x)=\exp(-x^2)\exp(-\mathrm{i}\kappa x)$, where $\kappa\geq0$. This equation was studied in \cite{Huang2017} in the context of beam propagation and the variables $x$ and $t$ denote the normalized transverse and longitudinal coordinates, respectively. In our simulations, we use the MTF spectral Galerkin method with $N=500$ and the scaling parameter $\nu=4$ in space, combined with the splitting method SM3 with $\tau=0.001$ in time. The top row of Figure \ref{fig:Beam} shows the evolution of $|\psi(x,t)|$ for $\kappa=0$. We see that $|\psi(x,t)|$ splits and is diffraction-free in the case of $\alpha=1$ and splits and diffracts in the case of $\alpha\in(1,2)$. Moreover, we also see that the diffraction of the beam becomes stronger as $\alpha$ increases. The bottom row of Figure \ref{fig:Beam} shows the evolution of $|\psi(x,t)|$ for $\kappa=10$, which indicates that the initial data can be viewed as an oblique Gaussian beam. We see that $|\psi(x,t)|$ exhibits diffraction-free propagation in the case of $\alpha=1$ and exhibits diffraction in the case of $\alpha\in(1,2)$. Moreover, similar to the previous case, the diffraction of the beam becomes stronger as $\alpha$ increases. Our results are consistent with the observations in \cite{Huang2017}.

\begin{figure}
\centering
\includegraphics[height=4.5cm,width=4.6cm]{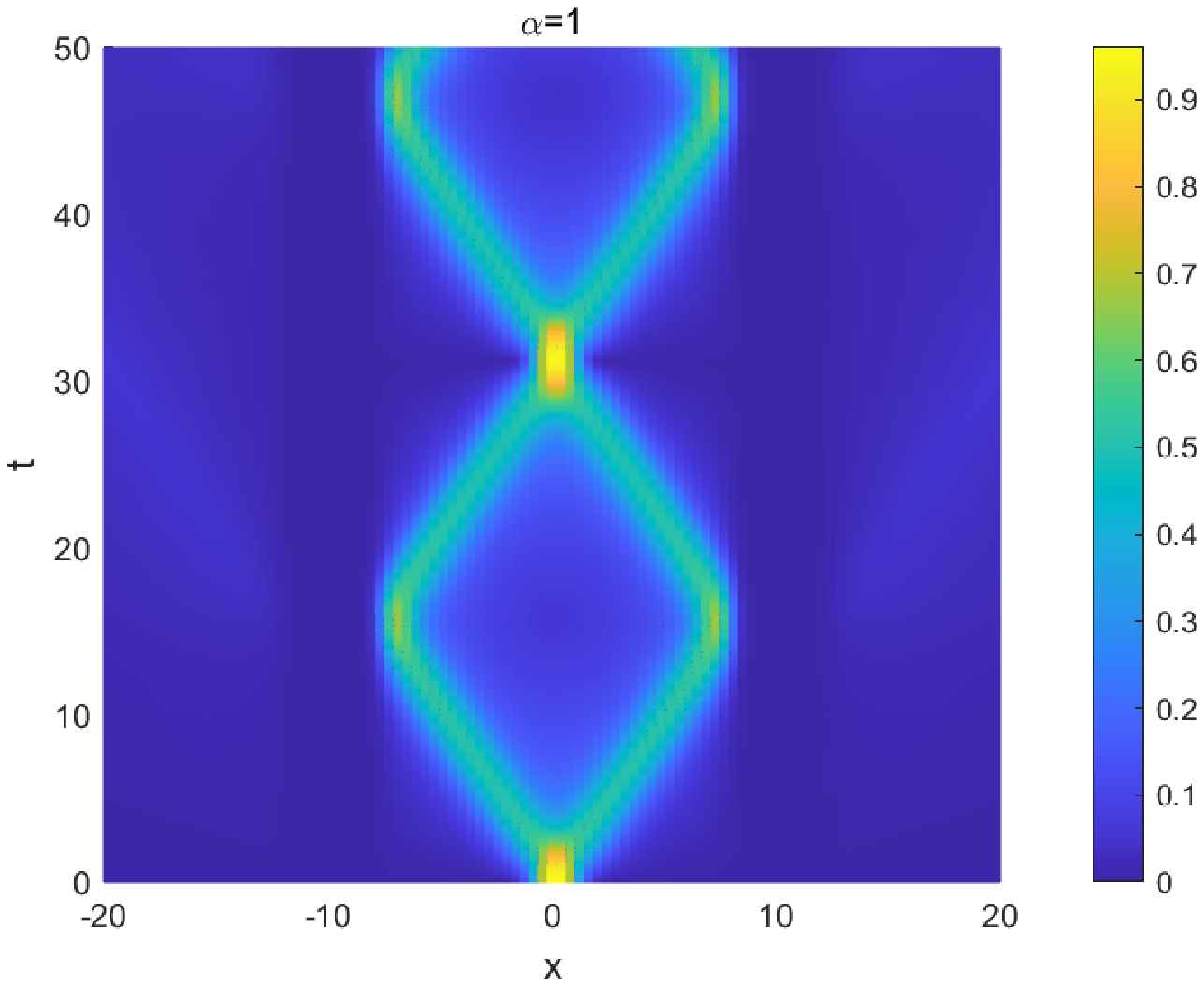}
\includegraphics[height=4.5cm,width=4.6cm]{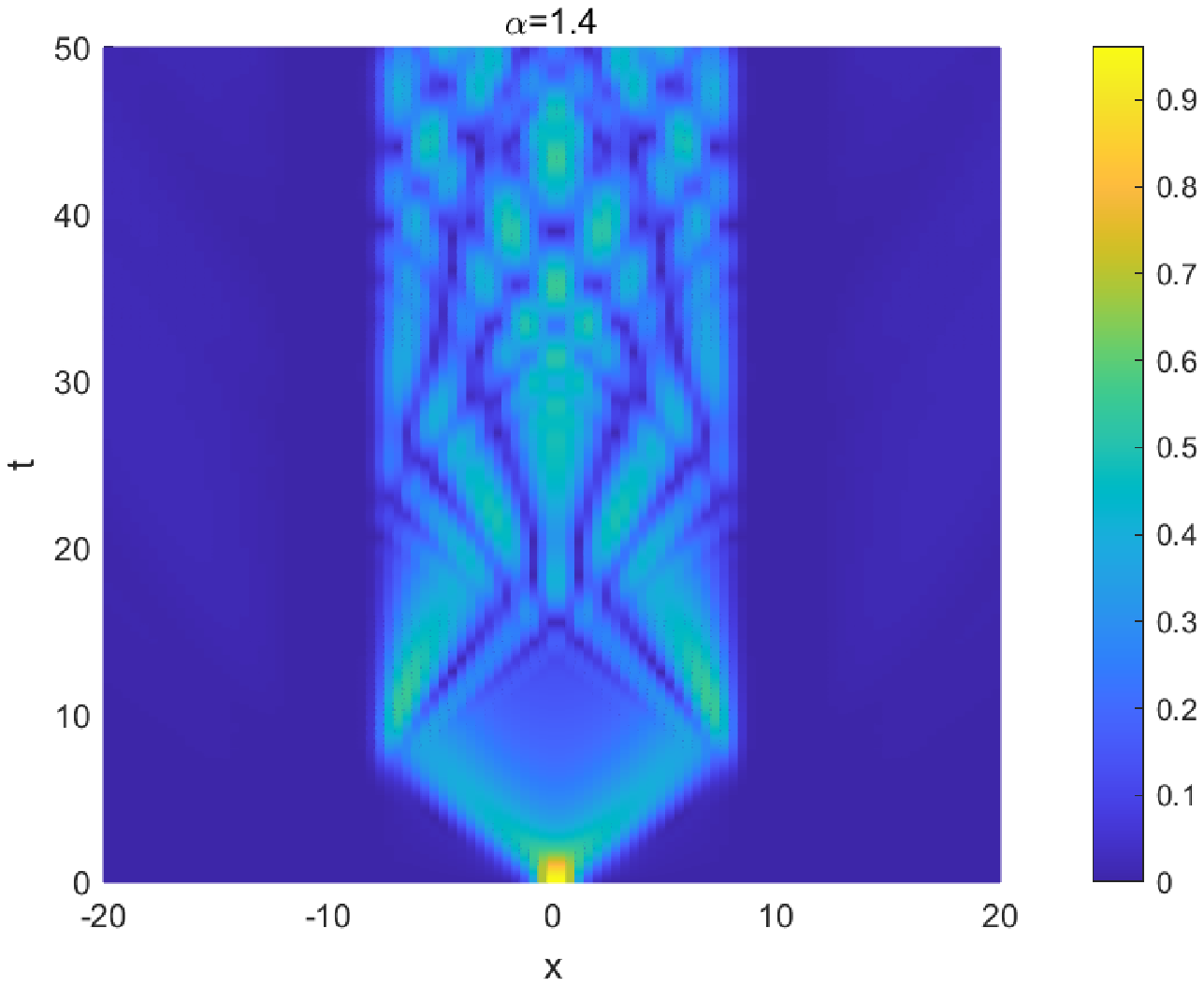}
\includegraphics[height=4.5cm,width=4.6cm]{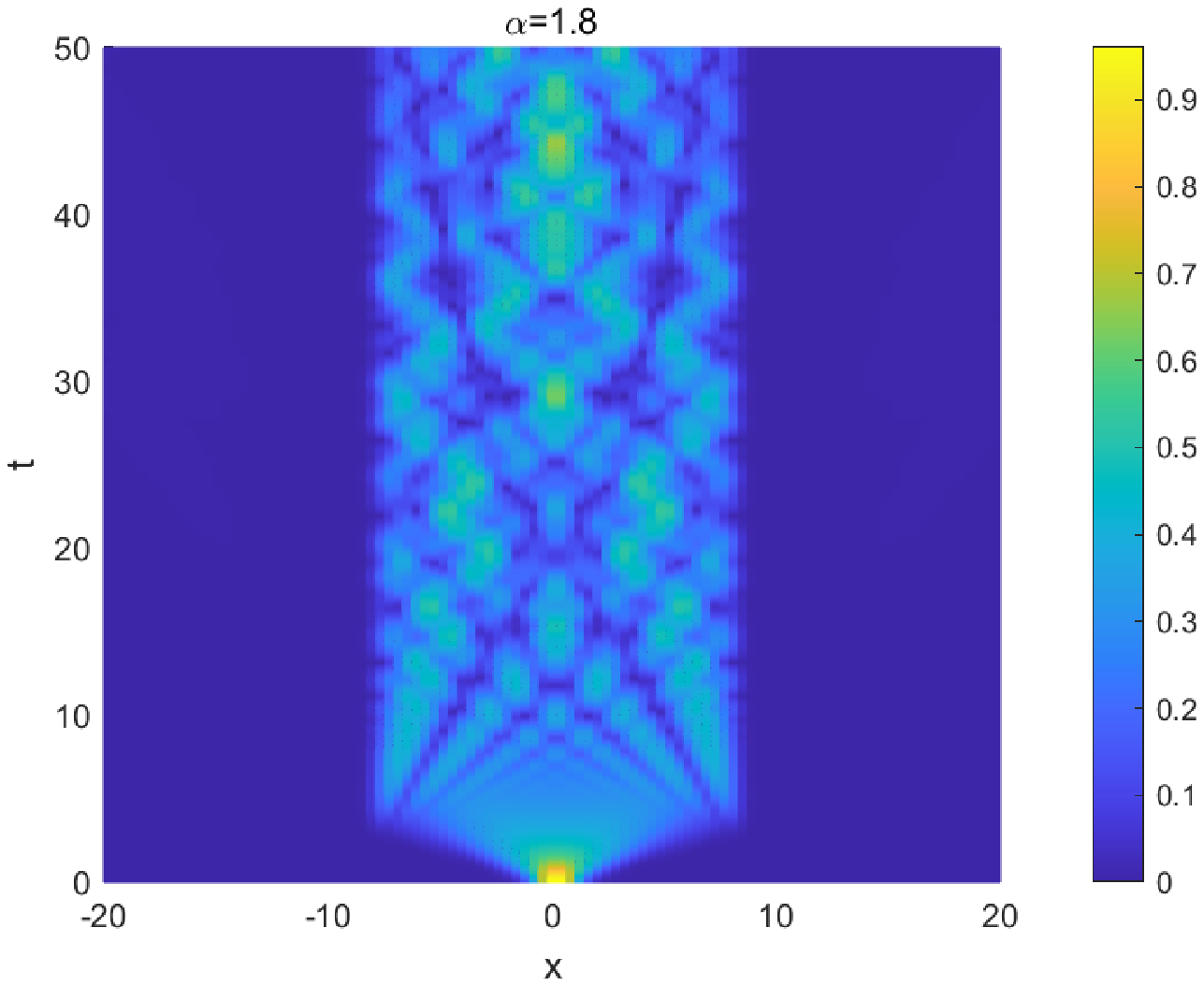} \\
\includegraphics[height=4.5cm,width=4.6cm]{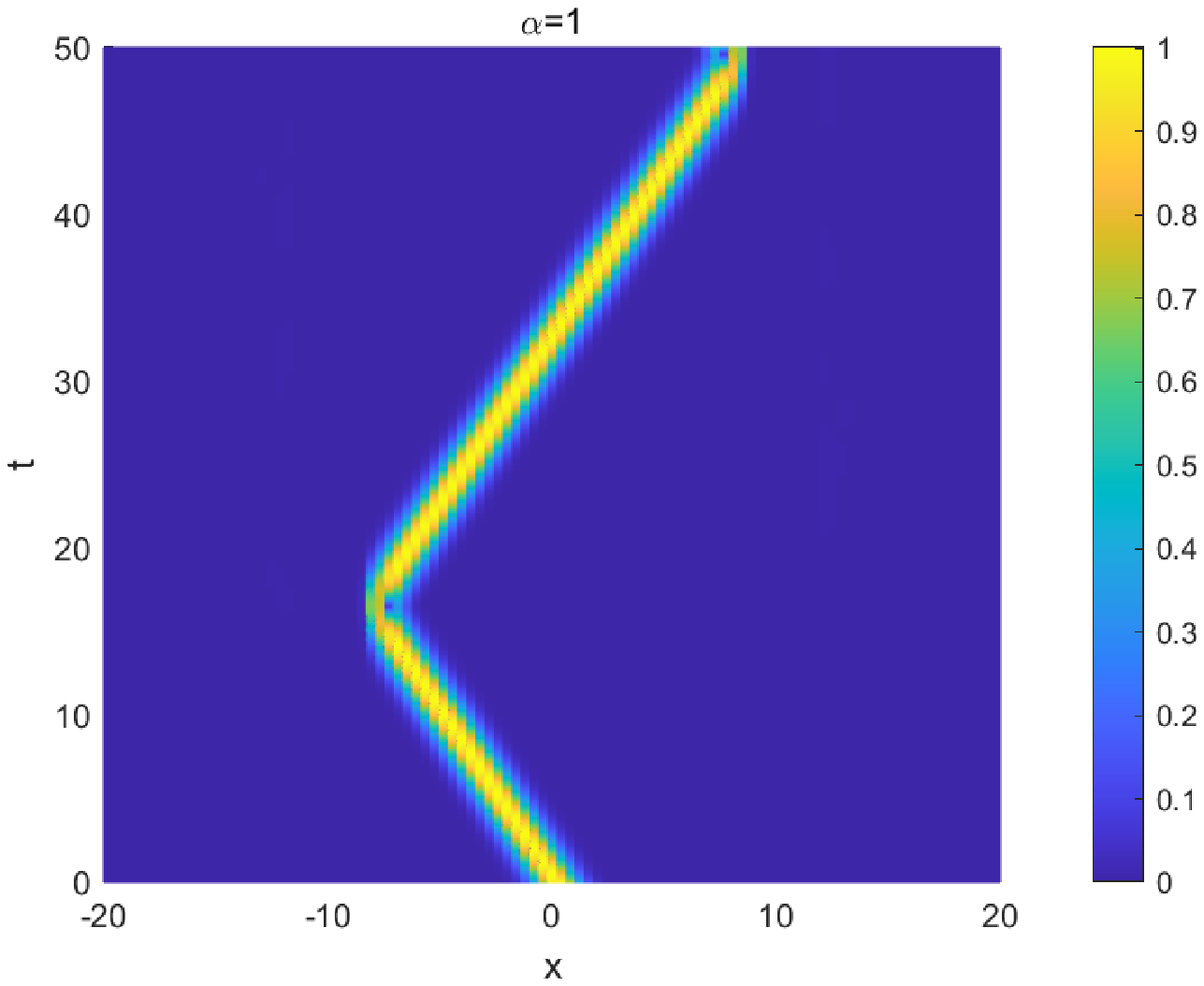}
\includegraphics[height=4.5cm,width=4.6cm]{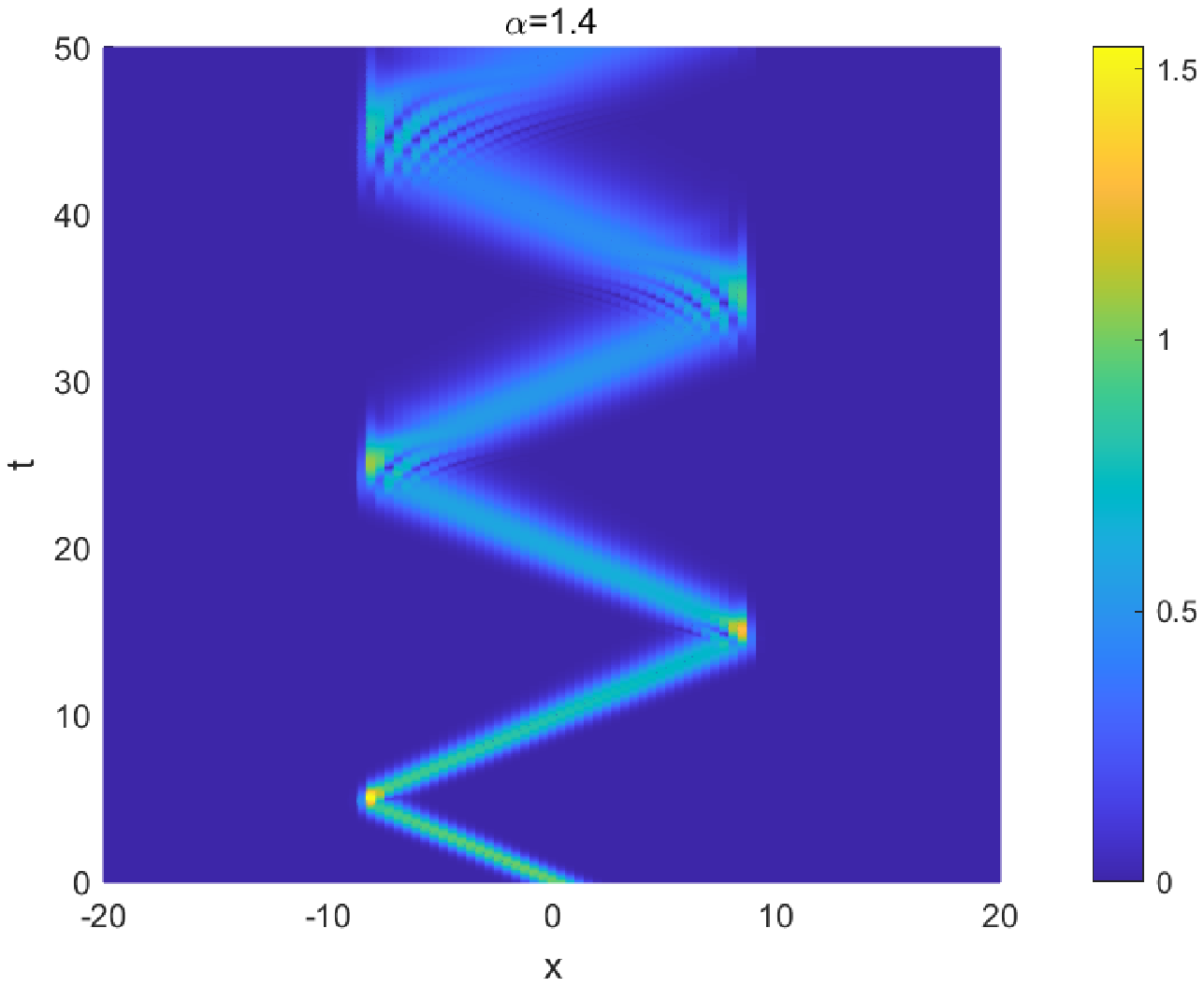}
\includegraphics[height=4.5cm,width=4.6cm]{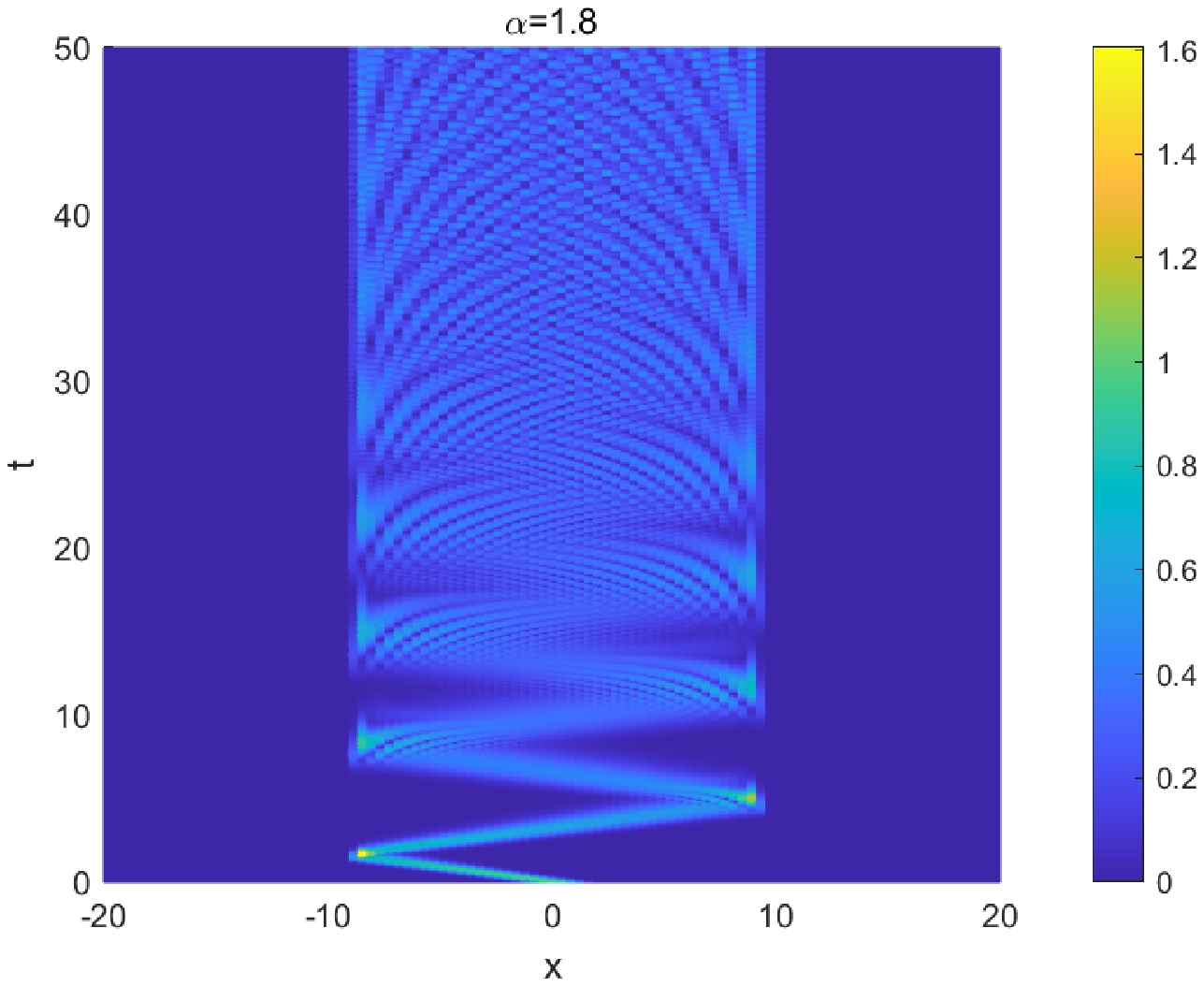}
\caption{Time evolution of $|\psi(x,t)|$ for the initial data $\psi_{0}(x)=\exp(-x^2)\exp(-\mathrm{i}\kappa x)$. Here $\kappa=0$ (top row) and $\kappa=10$ (bottom row).}
\label{fig:Beam}
\end{figure}

\vspace{.2cm}

\noindent{\bf Example 2}.
Let us consider the following focusing nonlinear FSE (see, e.g., \cite{Cayama2020ANM,Duo2016,Klein2014})
\begin{equation}\label{eq:nonlinearFSE}
\mathrm{i}\partial_t\psi(x,t) = \gamma (-\Delta)^{\alpha/2}\psi(x,t) - |\psi(x,t)|^{2}\psi(x,t),
\end{equation}
where $\gamma=1/2$. It is known that the solution satisfies the mass conservation, i.e.,
\begin{equation}\label{masstime}
M(t) = \int_{\mathbb{R}}|\psi(x,t)|^2 \mathrm{d}x = M(0).
\end{equation}
We discretize the equation \eqref{eq:nonlinearFSE} using the MTF spectral Galerkin method in space coupled with the Krogstad-P22 scheme \eqref{eq:ETDRKPade} in time. The nonlinear terms in \eqref{eq:ETDRKPade} are computed as follows: We only consider the term $\mathcal{N}(U_n,t_n)$ since the other terms $\mathcal{N}(a_n,t_n+\tau/2)$, $\mathcal{N}(b_n,t_n+\tau/2)$ and $\mathcal{N}(c_n,t_n+\tau)$ can be computed similarly. Firstly, let $\{\varrho_k\}_{k\in\mathbb{Z}}$ be the sequence defined by
\begin{equation}\label{def:varrho}
\varrho_k = \frac{\mathrm{i}^{-k}}{\pi^2} \int_{-\pi}^{\pi} \bigg|\cos\left(\frac{\theta}{2}\right)
\sum_{j=-N}^{N-1}\zeta_j(t_n)\mathrm{i}^je^{\mathrm{i}j\theta} \bigg|^2e^{-\mathrm{i}k\theta} \mathrm{d}\theta.
\end{equation}
From \eqref{def:AB} one can verify that $\mathcal{N}(U_n,t_n)=\mathcal{B}(U_n)U_n$, where $\mathcal{B}(U_n)\in\mathbb{C}^{2N\times2N}$ is defined by
\begin{equation}\label{def:N}
\mathcal{B}(U_n) = \left(
    \begin{array}{cccc}
       \varrho_0 & \varrho_{-1} & \cdots & \varrho_{1-2N} \\
       \varrho_1 & \varrho_0    & \cdots & \varrho_{2-2N} \\
       \vdots & \vdots  & \ddots & \vdots     \\
       \varrho_{2N-1} & \varrho_{2N-2} & \cdots  & \varrho_0
    \end{array}
    \right).
\end{equation}
It is clear that $\mathcal{B}(U_n)$ is a Toeplitz and Hermitian matrix. Furthermore, note that the integrand on the right-hand of \eqref{def:varrho} is periodic in $\theta$ with period $2\pi$, hence the elements of $\mathcal{B}(U_n)$ (i.e., $\{\varrho_k\}_{k=1-2N}^{2N-1}$) can be computed rapidly with the FFT. On the other hand, since $\mathcal{B}(U_n)$ is a Toeplitz matrix, $\mathcal{N}(U_n,t_n)=\mathcal{B}(U_n)U_n$ can also be computed by the FFT \cite{Golub2013}.

We first consider the initial data $\psi_0(x)=\mathrm{sech}(x)$. The top row of Figure \ref{fig:blowup} shows the time evolution of the solution $|\psi(x,t)|^2$ for the initial data $\psi_0(x)=\mathrm{sech}(x)$. Note that the focusing nonlinear equation \eqref{eq:nonlinearFSE} will display the finite time blow-up phenomenon whenever $\alpha\in(1/2,1]$ (see \cite{Klein2014}). In our simulations, we take $N=300,\nu=10,\tau=0.001$. Clearly, the finite time blow-up phenomenon is observed for $\alpha=0.8,1$. We then consider the initial data $\psi_0(x)=\mathrm{exp}(-x^2)$. In this case, the mass can be calculated as
\begin{equation}\label{massinit}
M(0) = \int_{\mathbb{R}} \mathrm{exp}(-2x^2) \mathrm{d}x = \sqrt{\frac{\pi}{2}}.
\end{equation}
The bottom row of Figure \ref{fig:blowup} shows the evolution of the solution $|\psi(x,t)|^2$ for different values of $\alpha$ and we take $N=300,~\nu=4,~\tau=0.001$ in our simulations. Again, the finite time blow-up phenomenon is observed for $\alpha=0.6,0.8$. Now, we turn our attention to the property of mass conservation. Using the change of variable $x =\tan(\theta/2)/2$, the mass error can be written as
\begin{align}
\mathrm{err}(t) = \left| \int_{\mathbb{R}} |\psi_N(x,t)|^2 \mathrm{d}x - M(0) \right|
= \bigg| \frac{1}{2\pi} \int_{-\pi}^{\pi} \bigg|\sum_{k=-N}^{N-1}\zeta_k(t)\mathrm{i}^ke^{\mathrm{i}{k}\theta} \bigg|^2 \mathrm{d}\theta -\sqrt{\frac{\pi}{2}} \bigg|,  \notag
\end{align}
and the integral in the last equation can be evaluated by the inverse FFT. In Figure \ref{fig:masserror} we plot the mass error of the MTF spectral method coupled with the Krogstad-P22 scheme for $\alpha=1.99$ and we choose $N=150,\nu=4,\tau=0.001$. We see that the mass errors are of the order $10^{-13}$ for $t\in[0,50]$, which indicates that our method is mass conserved. Moreover, compared with the mass error in \cite[Figure~12]{Cayama2020ANM}, which achieves the order $10^{-7}$ with the same number of spectral discretizations, our method is clearly much better.

\begin{figure}
\centering
\includegraphics[height=4.5cm,width=4.6cm]{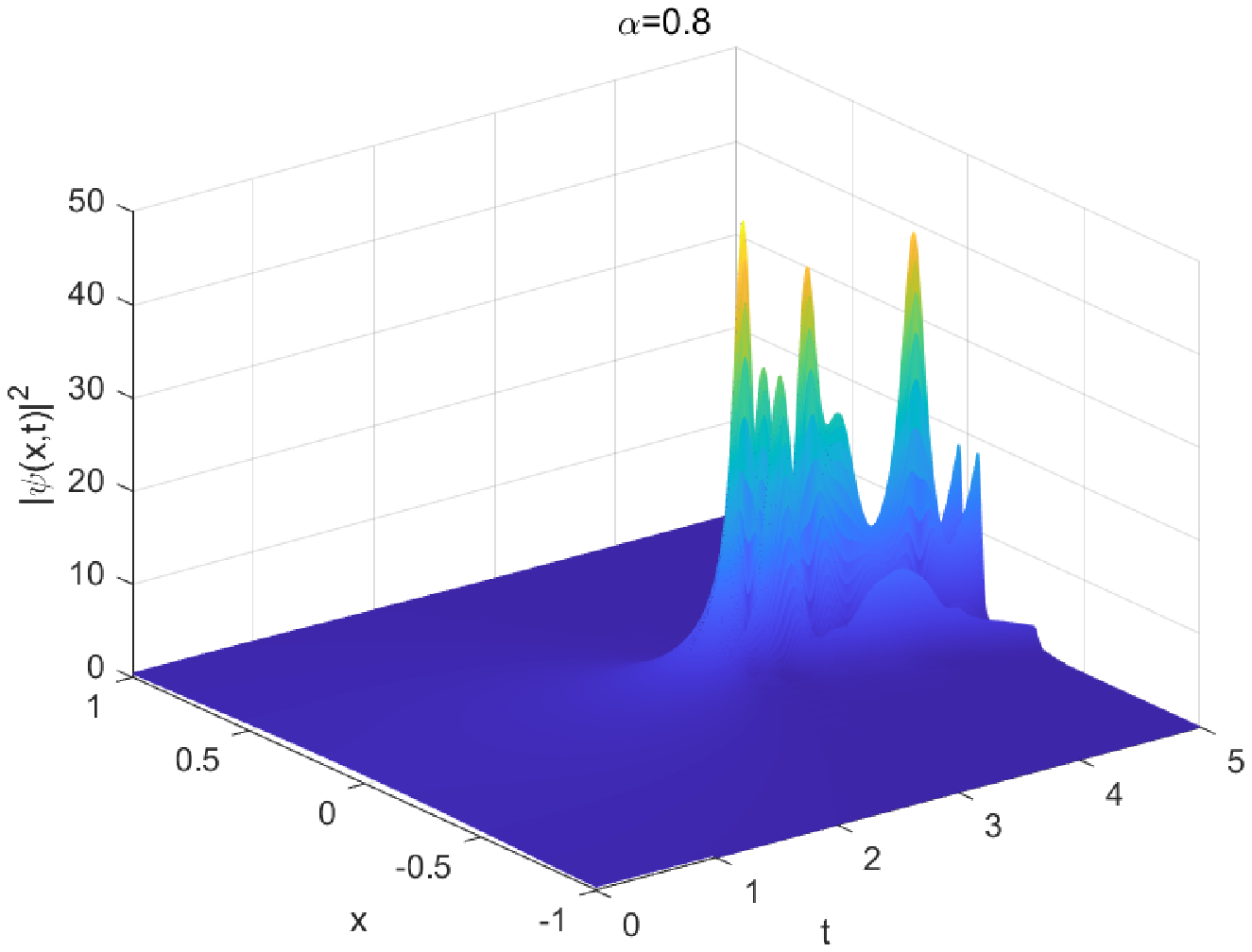}
\includegraphics[height=4.5cm,width=4.6cm]{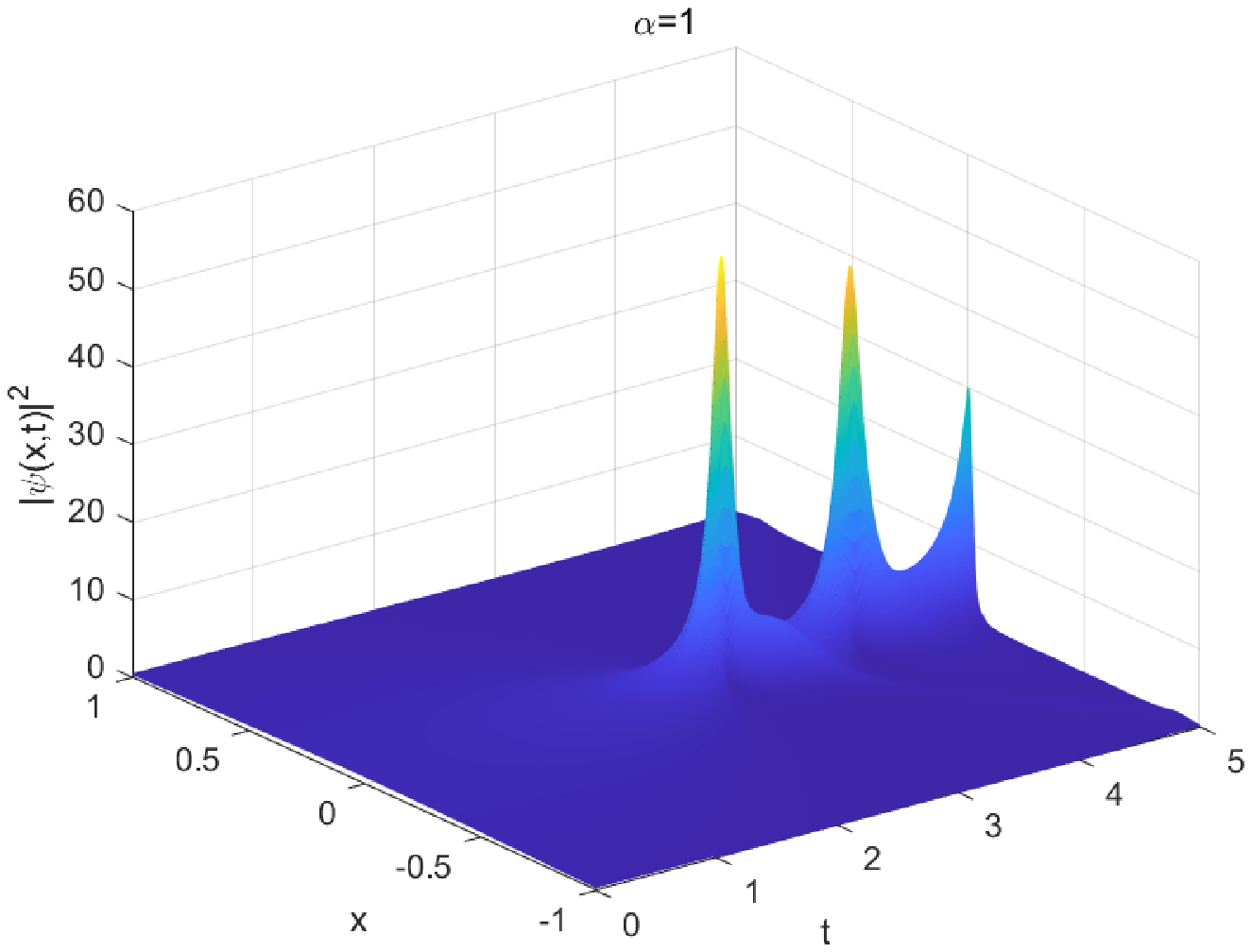}
\includegraphics[height=4.5cm,width=4.6cm]{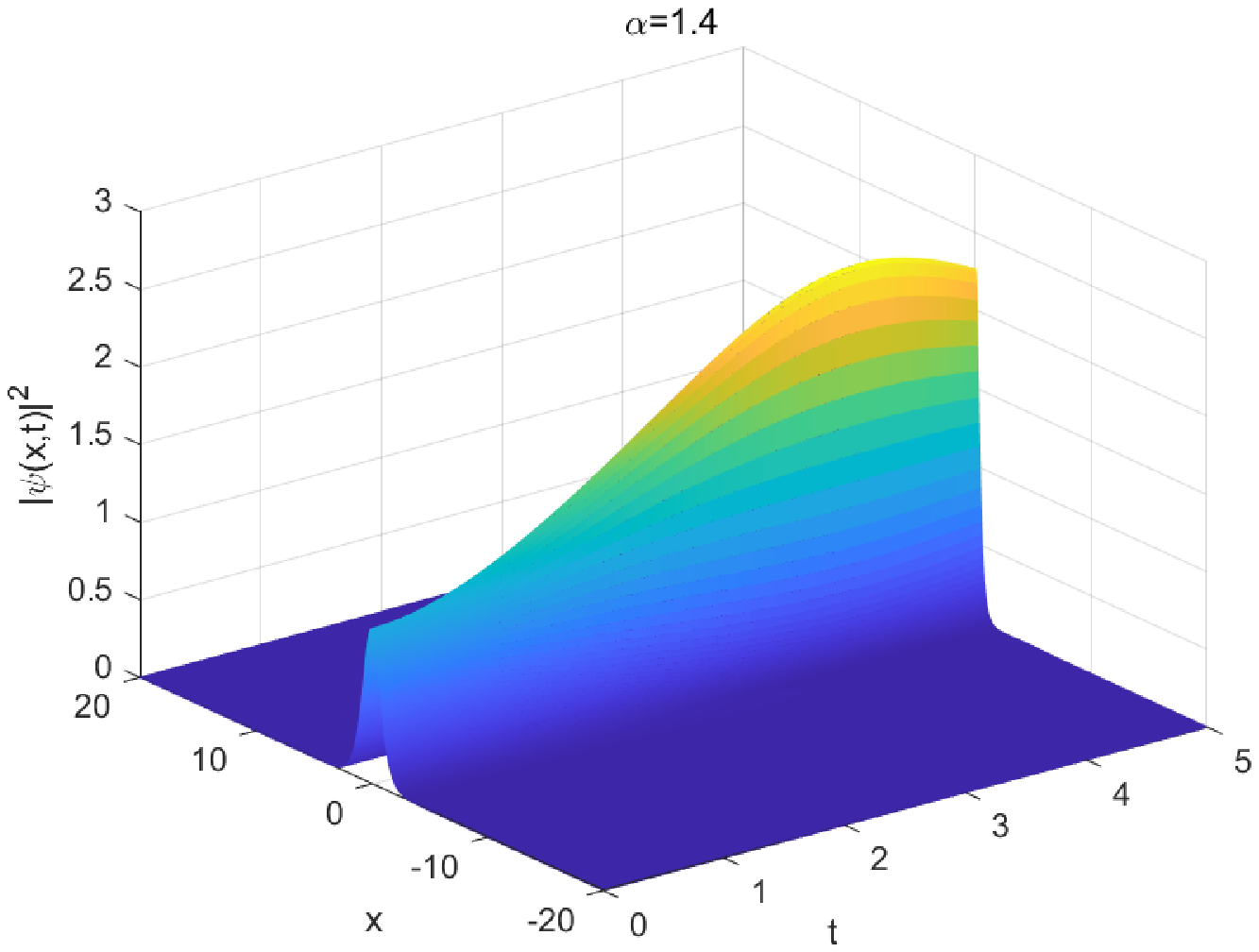}
\includegraphics[height=4.5cm,width=4.6cm]{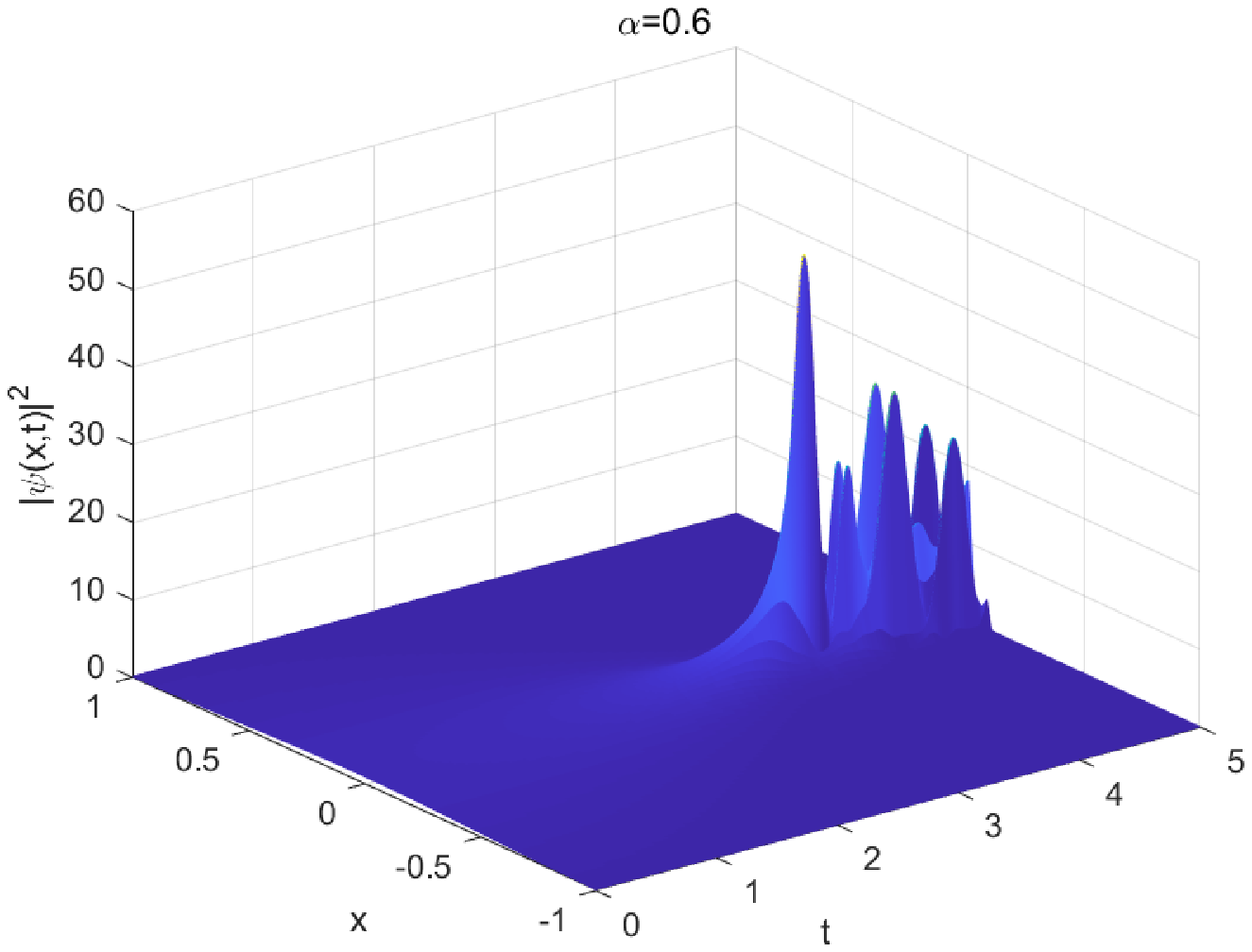}
\includegraphics[height=4.5cm,width=4.6cm]{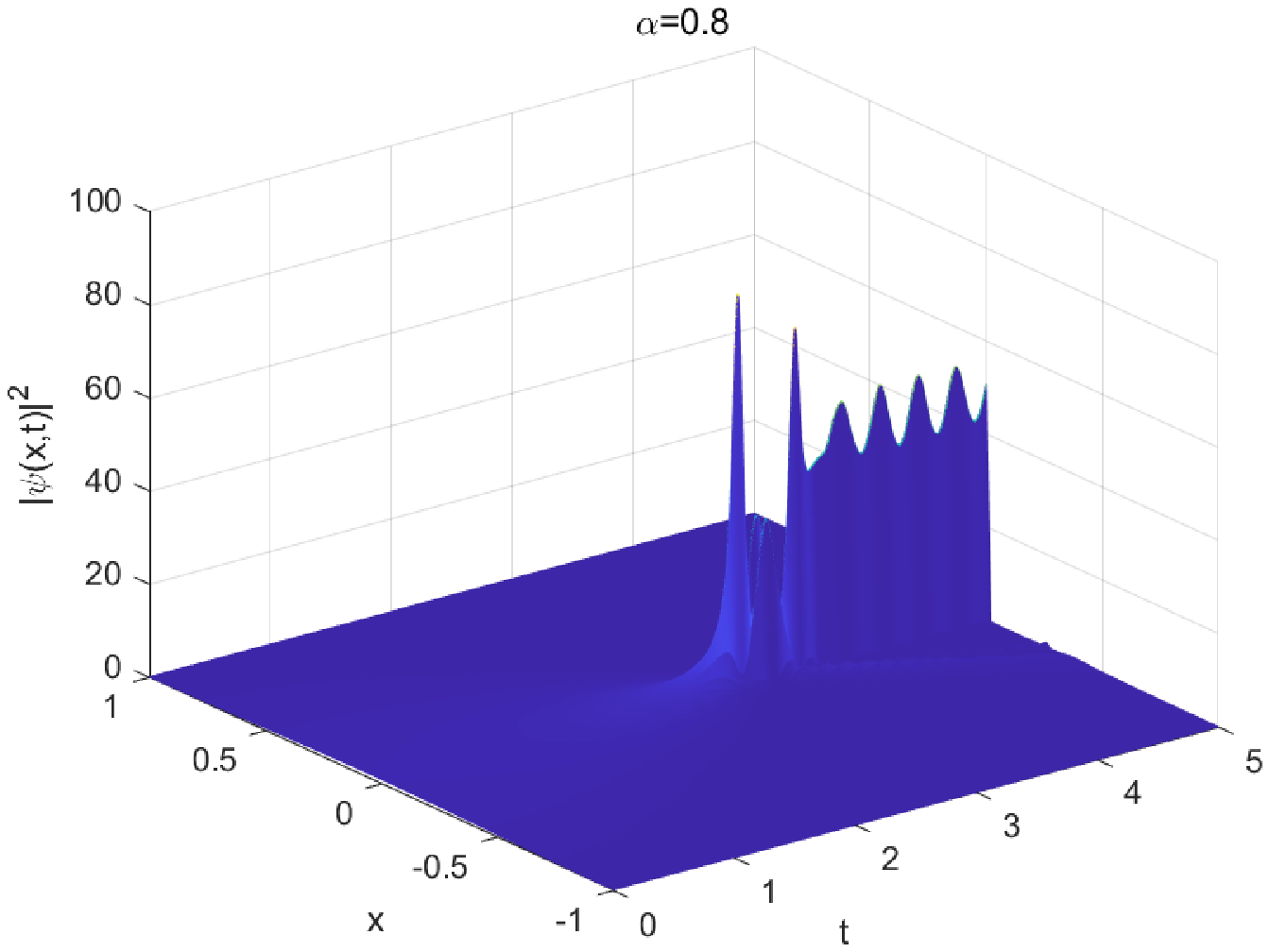}
\includegraphics[height=4.5cm,width=4.6cm]{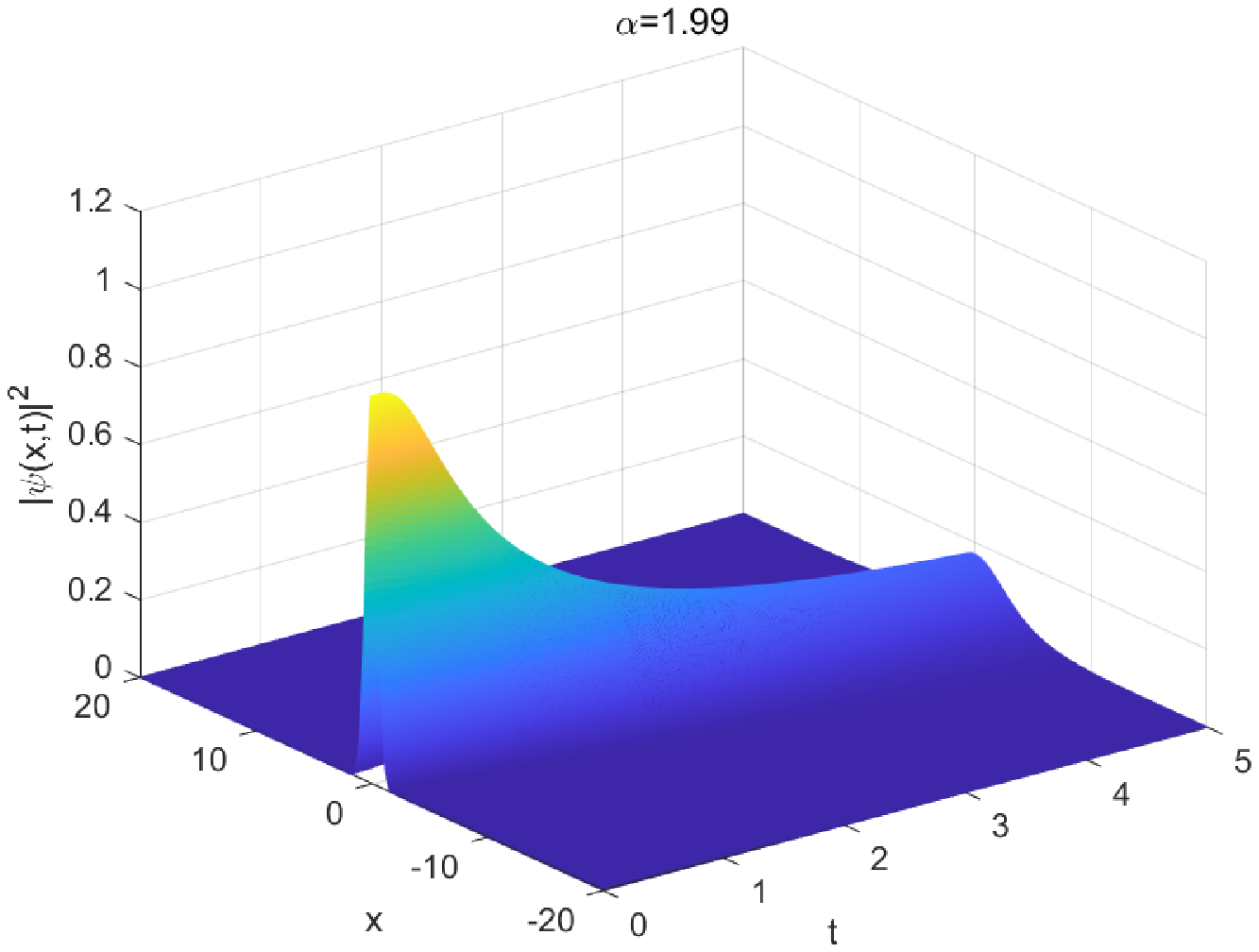}
\caption{Time evolution of $|\psi(x,t)|^2$ in \eqref{eq:nonlinearFSE} with the initial data $\psi_0(x)=\mathrm{sech}(x)$ (top row) and $\psi_0(x)=\mathrm{exp}(-x^2)$ (bottom row).  }
\label{fig:blowup}
\end{figure}

\begin{figure}
\centering
\includegraphics[height=6cm,width=8cm]{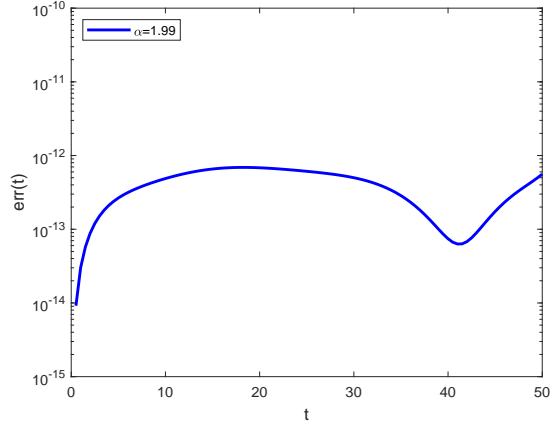}
\caption{Mass error of the MTF spectral Galerkin method coupled with the Krogstad-P22 scheme to \eqref{eq:nonlinearFSE} for $\alpha=1.99$. }
\label{fig:masserror}
\end{figure}

Finally, we illustrate the accuracy of the MTF spectral discretization in space and the temporal order of convergence of the Krogstad-P22 scheme \eqref{eq:ETDRKPade} in time. The reference solution is computed by using the MTF spectral Galerkin method coupled with \eqref{eq:ETDRKPade} with $N=300$ and $\tau=10^{-4}$. In Figure \ref{fig:spaceerror} we plot the maximum error of MTF spectral method with the scaling parameter $\nu=4$ coupled with the Krogstad-P22 scheme \eqref{eq:ETDRKPade} at time $t=1$. We can see that, similar to the linear case, the MTF spectral discretization converges at an exponential rate in the case $\alpha=1$ and at an algebraic rate in the case $\alpha\neq{1}$. In the left graph of Figure \ref{fig:NFSEtimeab}, we plot the maximum error of the MTF spectral method with $\nu=4$ coupled with \eqref{eq:ETDRKPade} at time $t=1$ as a function of the time step size $\tau$. We see that the temporal order of convergence is four for all choices of $\alpha$. In the right graph of Figure \ref{fig:NFSEtimeab}, we plot the asymptotic behavior of the computed solutions at time $t=1$ for three values of $\alpha$. We see that all computed solutions decay at the rate $\mathcal{O}(|x|^{-\alpha-1})$, which are in agreement with the numerical observation in \cite{Tang2020}.

\begin{figure}
\centering
\includegraphics[height=5cm,width=7cm]{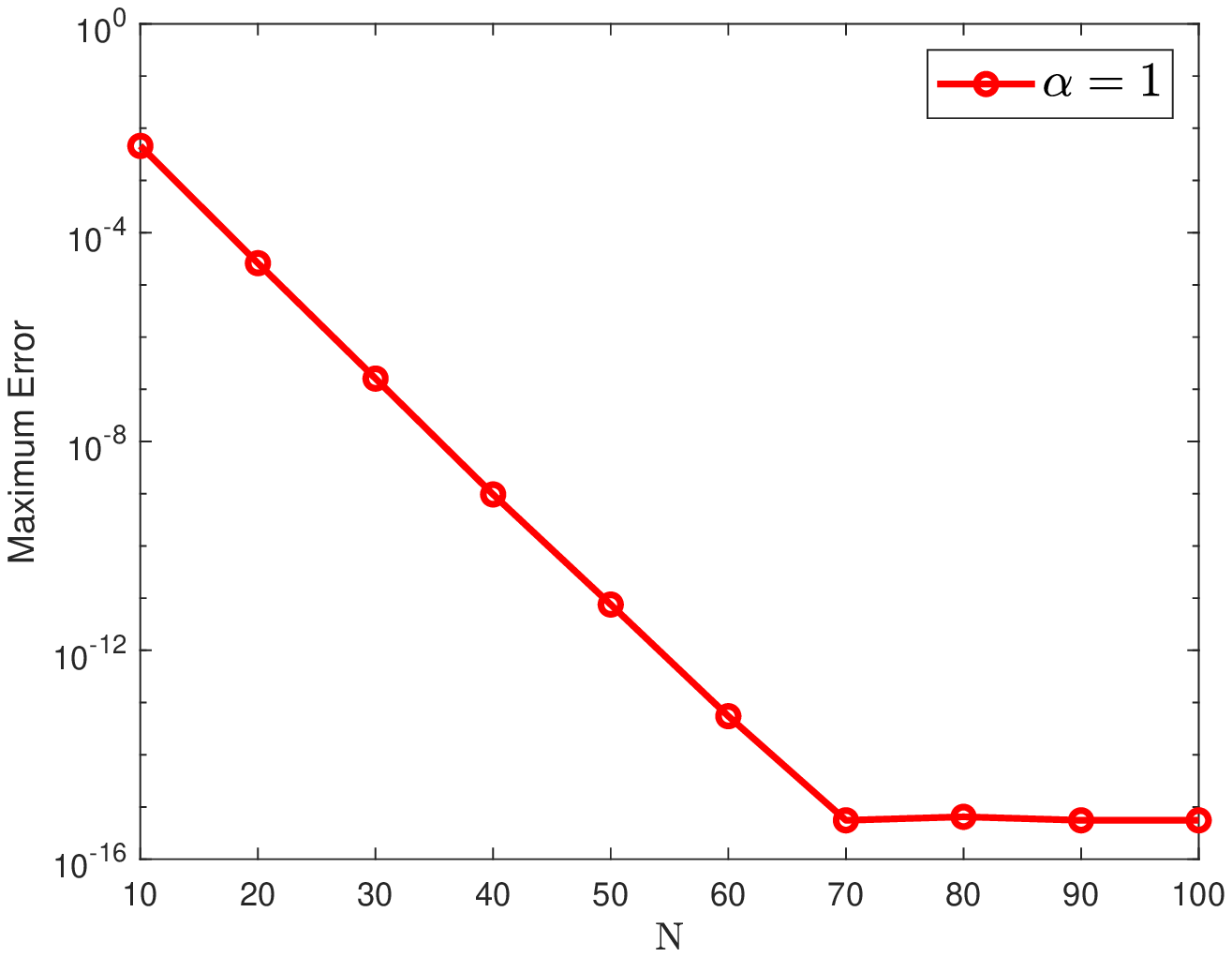}
\includegraphics[height=5cm,width=7cm]{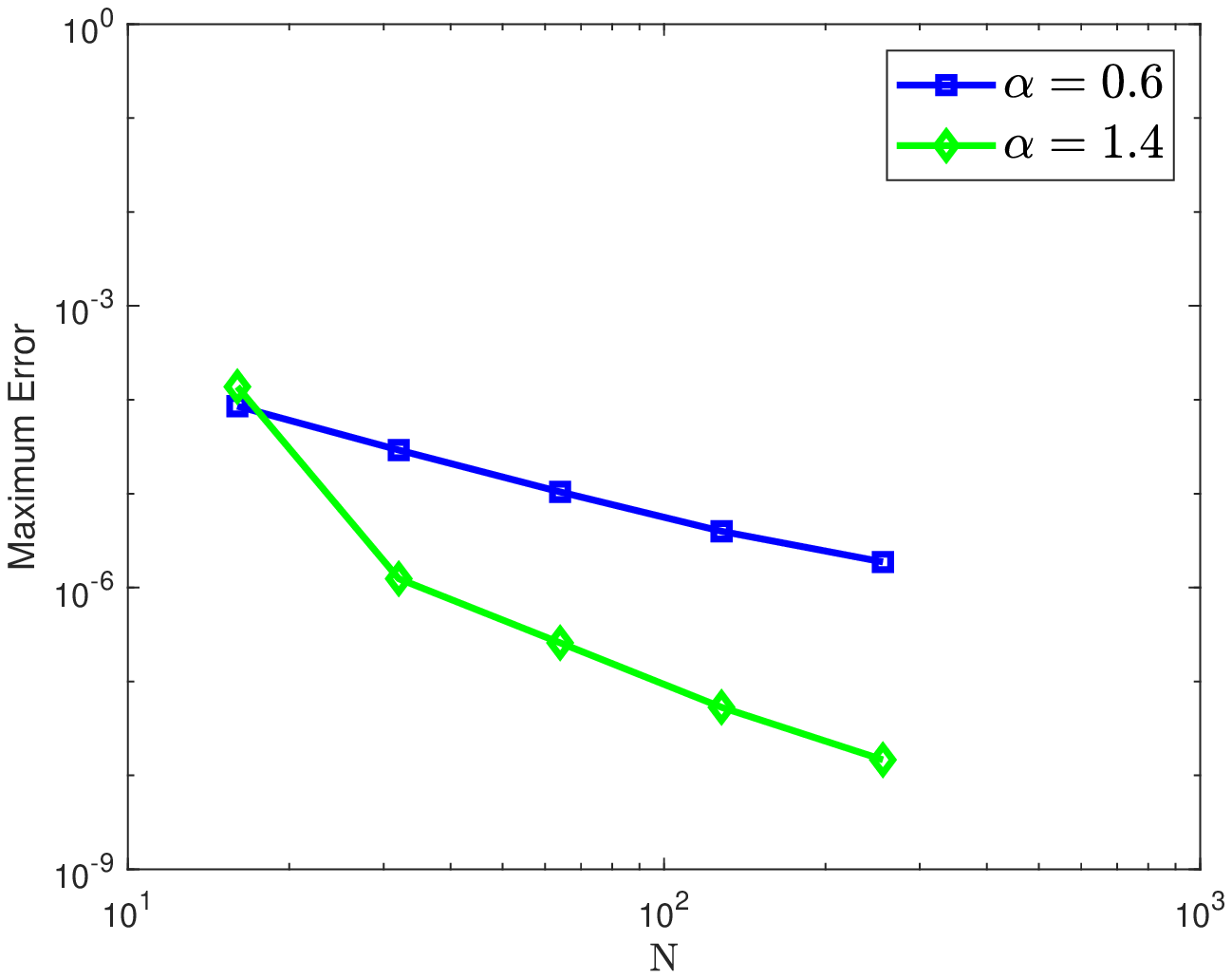}
\caption{Maximum errors of the MTF spectral method coupled with the Krogstad-P22 scheme at time $t=1$ as a function of $N$. }
\label{fig:spaceerror}
\end{figure}

\begin{figure}
\centering
\includegraphics[height=5cm,width=7.0cm]{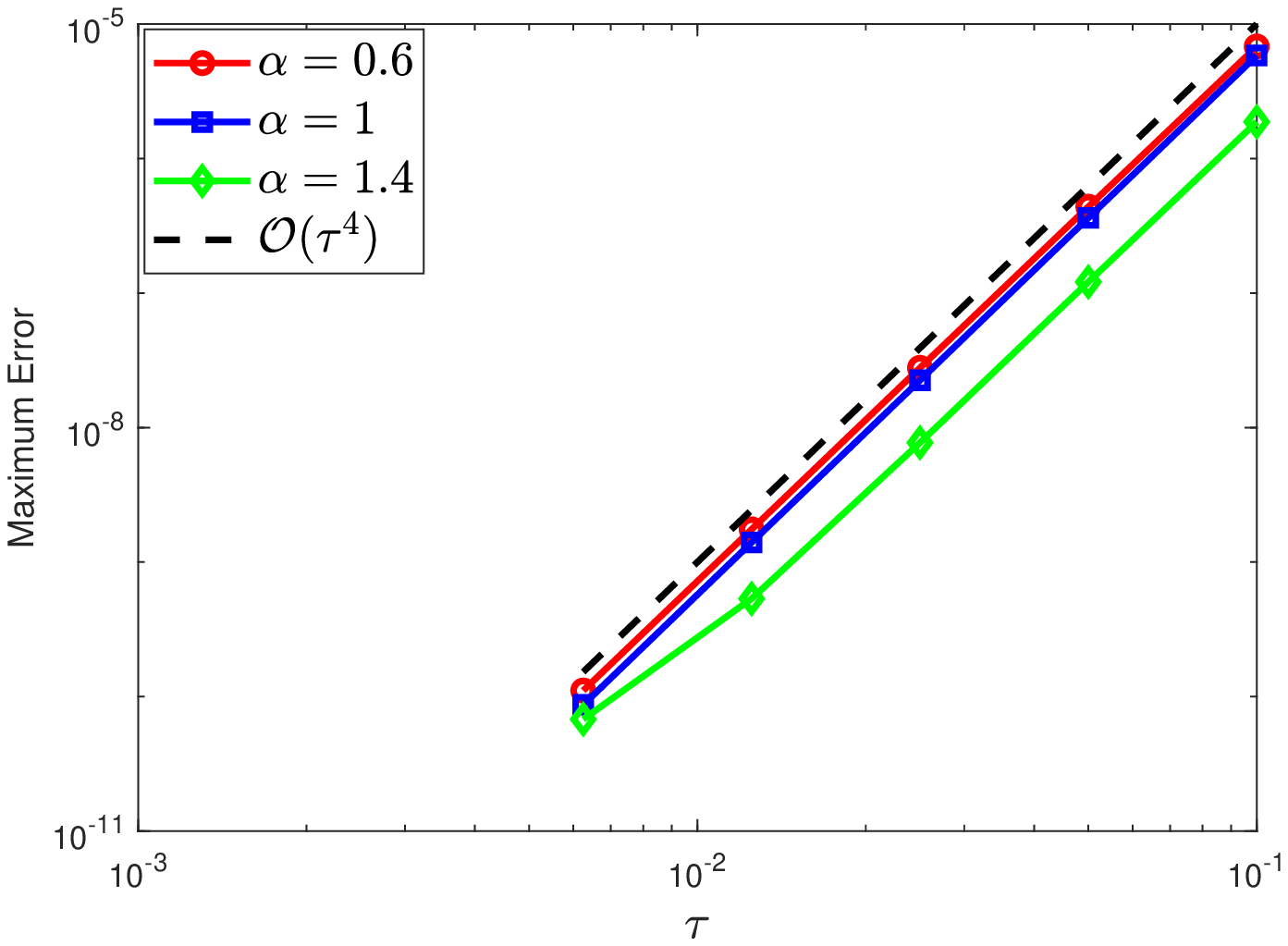}
\includegraphics[height=5cm,width=7.0cm]{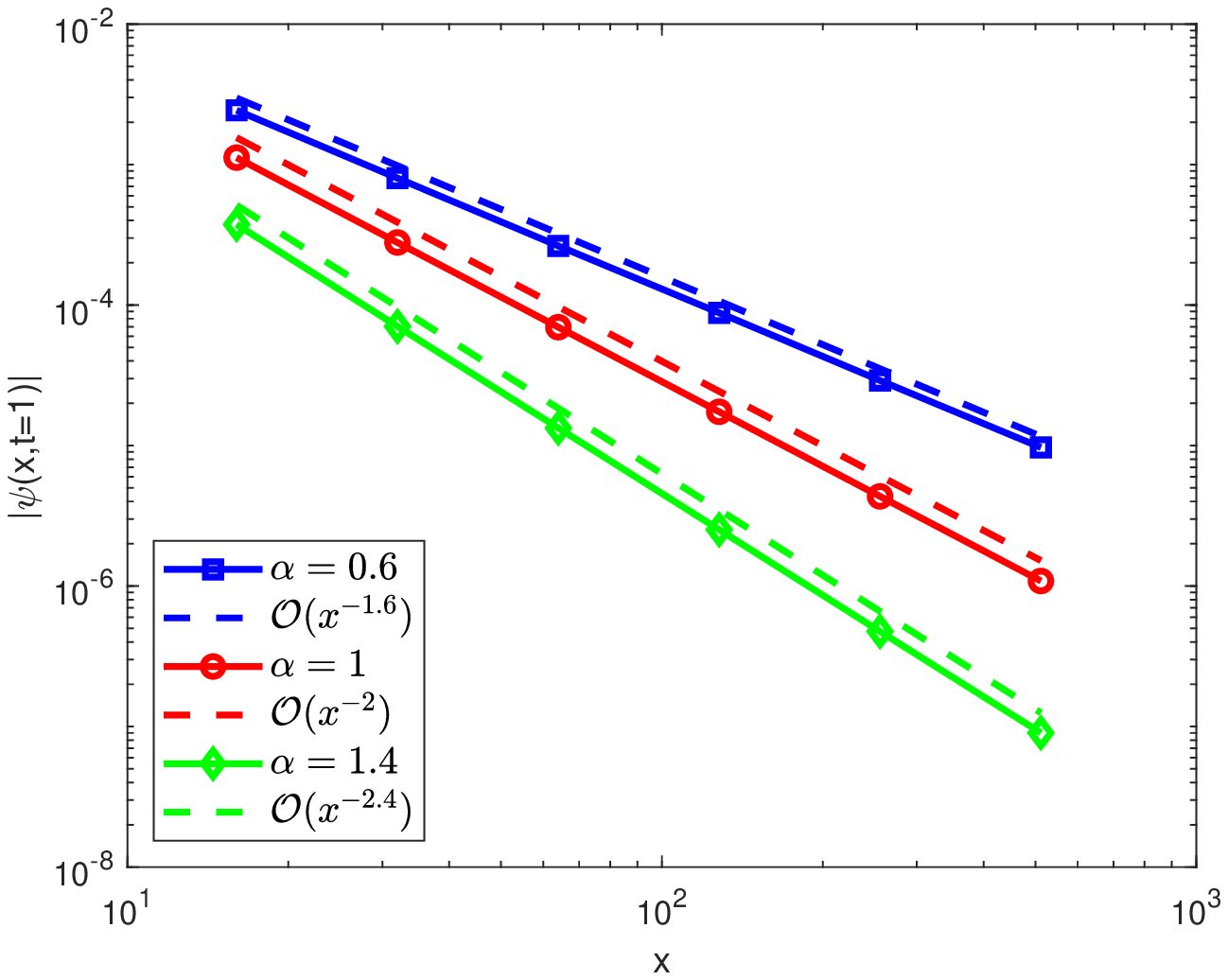}
\caption{Left: Maximum errors of MTF spectral method coupled with the Krogstad-P22 scheme \eqref{eq:ETDRKPade} at time $t=1$ as a function of $\tau$. Right: Asymptotic behavior of the solution $\psi(x,t=1)$ for three values of $\alpha$.}
\label{fig:NFSEtimeab}
\end{figure}

\section{Conclusions}\label{sect:Con}
In this paper, we have studied the numerical solution of the FSE on the real line. We proposed a new spectral discretization using MTFs in space combined with some time-stepping methods for the discretization in time. This new spectral discretization can achieve exponential convergence in space in the case of $\alpha=1$, regardless of the underlying FSE is linear or nonlinear, and exhibits a comparable or even better performance than state-of-the-art spectral discretization schemes in other cases. We conclude that spectral methods using MTFs are competitive for solving PDEs whose solution has slow decay behavior at infinity.

Before closing this paper, we list several problems for future research:
\begin{itemize}
\item Our study in this work is restricted to the one-dimensional FSE problem. However, an extension of the current work to the high dimensional FSE problem is straightforward. An issue arised in the process of extension is a multivariate counterpart of the integral in \eqref{eq:AStepI}, which might be difficult to evaluate due to the singular and nonseparable factor $|\xi|^{\alpha}$.

\item Both MTFs and MCFs are orthogonal systems on the real line and spectral approximations using them can be achieved rapidly by using the FFT. An interesting problem is to compare their approximation powers for functions with exponential or algebraic decay behavior at infinity.

\item Conservation laws, such as mass and energy, are important in quantum mechanics. In the case of the nonlinear FSE, Figure \ref{fig:masserror} implies that the MTF spectral Galerkin method coupled with the Krogstad-P22 scheme is mass conserved. However, a rigorous analysis of this observation is still lacking.
\end{itemize}
We will address these issues in the future.

\section*{Acknowledgements}
This work was supported by the National Natural Science Foundation of China under
grant 11671160. The authors are grateful to Prof. Chengming Huang for his valuable suggestions in improving this paper.

\end{document}